\documentclass[12pt,lenq]{amsart}
\usepackage{graphicx}
\usepackage{amsmath}
\usepackage{amscd}
\usepackage{amssymb}
\usepackage{amsbsy}
\usepackage{amsfonts}
\usepackage{latexsym}
\usepackage{graphics}
\usepackage{amsmath,amscd,latexsym}
\usepackage{multirow}
\usepackage{array}
\usepackage{paralist}
\usepackage{titletoc}
\theoremstyle{plain}
\newtheorem{theorem}{Theorem}[section]

\newtheorem{lemma}[theorem]{Lemma}
\newtheorem{corollary}[theorem]{Corollary}
\newtheorem{remark}[theorem]{Remark}

\newtheorem{notation}[theorem]{Notation}

\newtheorem{main theorem}[theorem]{Main Theorem}

\newtheorem{question}[theorem]{Question}
\newtheorem{convention}[theorem]{Convention}

\newtheorem*{hypothesisA}{Hypothesis A}
\newtheorem*{hypothesisB}{Hypothesis B}
\newtheorem*{hypothesisC}{Hypothesis C}

\newcommand{\ZZ}{\mathbb{Z}}

\newcommand{\RR}{\mathbb{R}}
\newcommand{\CC}{\mathbb{C}}

\newcommand{\PSL}{\mbox{$\mathrm{PSL}$}}

\newcommand{\svert}{\,|\,}

\newcommand{\lp}{(\hskip -0.07cm (}
\newcommand{\rp}{)\hskip -0.07cm )}

\begin{document}

\title{Homotopically equivalent simple loops
on 2-bridge spheres in 2-bridge link complements (II)}
\author{Donghi Lee}
\address{Department of Mathematics\\
Pusan National University \\
San-30 Jangjeon-Dong, Geumjung-Gu, Pusan, 609-735, Republic of Korea}
\email{donghi@pusan.ac.kr}

\author{Makoto Sakuma}
\address{Department of Mathematics\\
Graduate School of Science\\
Hiroshima University\\
Higashi-Hiroshima, 739-8526, Japan}
\email{sakuma@math.sci.hiroshima-u.ac.jp}

\subjclass[2010]{Primary 20F06, 57M25 \\
\indent {The first author was supported by Basic Science Research Program
through the National Research Foundation of Korea(NRF) funded
by the Ministry of Education, Science and Technology(2012R1A1A3009996).
The second author was supported
by JSPS Grants-in-Aid 22340013 and 21654011.}}

\begin{abstract}
This is the second of a series of papers which give a necessary and sufficient condition for two essential simple loops on a 2-bridge sphere in a 2-bridge link complement to be homotopic in the link complement. The first paper of the series treated the case of the $2$-bridge torus links. In this paper, we treat the case of $2$-bridge links of slope $n/(2n+1)$ and $(n+1)/(3n+2)$, where $n \ge 2$ is an arbitrary integer.
\end{abstract}
\maketitle

\section{Introduction}

Let $K$ be a 2-bridge link in $S^3$ and let $S$
be a 4-punctured sphere
in $S^3-K$ obtained from a $2$-bridge sphere of $K$.
In \cite{lee_sakuma}, we gave a complete characterization of
those essential simple loops in $S$ which are null-homotopic in $S^3-K$.
The purpose of this series of papers
starting from \cite{lee_sakuma_2} and ending with \cite{lee_sakuma_4},
including the present paper as the second one,
is to give a necessary and sufficient condition
for two essential simple loops on $S$
to be homotopic in $S^3-K$.
In the first paper~\cite{lee_sakuma_2}
of the series,
we treated the case
when the 2-bridge link is a $(2,p)$-torus link.
In this paper, we treat the case
of 2-bridge links of slope $n/(2n+1)$ and $(n+1)/(3n+2)$,
where $n \ge 2$ is an arbitrary integer.
These two families play special roles
in our project in the sense that
the treatment of these links form a base step
of an inductive proof of a theorem for general $2$-bridge links
giving an answer to the problem treated in this series of papers.
We note that the figure-eight knot is both
a 2-bridge link of slope $n/(2n+1)$ with $n=2$ and a 2-bridge link of slope $(n+1)/(3n+2)$ with $n=1$.
Surprisingly, the treatment of the figure-eight knot,
the simplest hyperbolic 2-bridge knot,
is the most complicated.
In fact, the figure-eight knot group admits
various unexpected reduced annular diagrams
(see Section~\ref{proof_of_corollary}).
This reminds us of the phenomenon
in the theory of exceptional
Dehn filling that
the figure-eight knot attains the
maximal number of exceptional Dehn fillings.

This paper is organized as follows.
In Section~\ref{statements}, we
describe the main results of this paper (Main Theorems~\ref{main_theorem} and \ref{main_theorem_2}).
In Section~\ref{sec:technical_lemmas},
we set up Hypotheses~A, B and C, under which we establish technical lemmas
used for the proofs in Sections~\ref{proof_for_twist_links_1}--\ref{proof_for_[2,1,n]}.
The special case of Main Theorem~\ref{main_theorem}
(namely, the case of a 2-bridge link of slope $2/5$)
is treated in Section~\ref{proof_for_twist_links_1}, and
the remaining case of Main Theorem~\ref{main_theorem}
(namely, the case of a 2-bridge link of slope $n/(2n+1)$ with $n \ge 3$)
in Section~\ref{proof_for_twist_links_2}.
The proof of Main Theorem~\ref{main_theorem_2}
is contained in Section~\ref{proof_for_[2,1,n]}.
In the final section, Section~\ref{proof_of_corollary},
we prove Theorems~\ref{main_corollary} and
\ref{main_corollary_2}.

\section{Main results}
\label{statements}

This paper, as a continuation of \cite{lee_sakuma_2},
uses the same notation and terminology as in
\cite{lee_sakuma_2} without specifically mentioning.
We begin with the following question,
providing whose answer is the purpose of this series of papers.

\begin{question} \label{question}
Consider a $2$-bridge link $K(r)$ with $r\ne \infty$.
For two distinct rational numbers $s, s' \in I_1(r) \cup I_2(r)$,
when are the unoriented loops $\alpha_s$ and $\alpha_{s'}$
homotopic in $S^3-K(r)$?
\end{question}

In the first paper~\cite{lee_sakuma_2},
we treated the case when $r=1/p$ for some $p\in\ZZ$,
and obtained a complete answer (see \cite[Main Theorem~2.7]{lee_sakuma_2}).
In the present paper, we solve
the above question for the $2$-bridge links $K(n/(2n+1))$
and $K((n+1)/(3n+2))$, where $n \ge 2$ is an arbitrary integer.

\begin{main theorem} \label{main_theorem}
Suppose $r=n/(2n+1)=[2, n]$, where $n \ge 2$ is an integer.
Then, for any two distinct rational numbers $s, s' \in I_1(r) \cup I_2(r)$,
the unoriented loops $\alpha_s$ and $\alpha_{s'}$ are never homotopic in $S^3-K(r)$.
\end{main theorem}

\begin{main theorem} \label{main_theorem_2}
Suppose $r=(n+1)/(3n+2)=[2, 1, n]$, where $n \ge 2$ is an integer.
Then, for two distinct rational numbers $s, s' \in I_1(r) \cup I_2(r)$,
the unoriented loops $\alpha_s$ and $\alpha_{s'}$ are homotopic in
$S^3-K(r)$
if and only if both $r=3/8$ (i.e., $n=2$)
and the set $\{s, s'\}$ equals
either $\{1/6, 3/10\}$ or $\{3/4, 5/12\}$.
\end{main theorem}

\begin{remark}
{\rm
The exceptional pairs $\{1/6, 3/10\}$ and $\{3/4, 5/12\}$
have the following geometric properties
in the Farey tessellation.
Let $\tau$ be the reflection of the hyperbolic plane in the geodesic
with endpoints $1/2$ and $1/4$,
which bisects the Farey edge
$\langle 0/1, 1/3\rangle$.
Then $\tau$ preserves the Farey tessellation and
interchanges $\infty$ and $r=3/8$.
The members of each of the exceptional pairs are
interchanged by the involution $\tau$.
}
\end{remark}

We prove the above main theorems by
interpreting the situation in terms of combinatorial group theory.
In other words, we prove that
two words representing the free homotopy classes of
$\alpha_s$ and $\alpha_{s'}$
are conjugate in the $2$-bridge link group $G(K(r))$ if and only if
$s$ and $s'$ satisfy the conditions given in the statements of the theorems.
The key tool used in the proofs is small cancellation theory,
applied to two-generator and one-relator presentations
of $2$-bridge link groups.
The proofs of the main theorems also imply the following theorems.

\begin{theorem}
\label{main_corollary}
Suppose $r=n/(2n+1)=[2, n]$, where $n \ge 2$ is an integer.
Then the following hold
for a rational number $s \in I_1(r) \cup I_2(r)$.
\begin{enumerate}[\indent \rm (1)]
\item The loop $\alpha_s$ is peripheral
if and only if one of the following holds.
\begin{enumerate}[\rm (a)]
\item $n=2$, i.e., $r=2/5$, and $s=1/5$ or $s=3/5$.

\item $s=(n+1)/(2n+1)$.
\end{enumerate}

\item The free homotopy class $\alpha_s$ is primitive
with the following exceptions.
\begin{enumerate}[\rm (a)]
\item $r=2/5$, and $s=2/7$ or $s=3/4$.
In this case, $\alpha_s$ is the third power of some primitive element
in $G(K(r))$.

\item $r=3/7$ and $s=2/7$.
In this case, $\alpha_s$ is the second power of some primitive element
in $G(K(r))$.
\end{enumerate}
\end{enumerate}
\end{theorem}

\begin{theorem}
\label{main_corollary_2}
Suppose $r=(n+1)/(3n+2)=[2, 1, n]$, where $n \ge 2$ is an integer.
Then the loop $\alpha_s$ is non-peripheral and primitive
for any rational number $s \in I_1(r) \cup I_2(r)$.
\end{theorem}

Here, a closed loop $\alpha_s$ in $S^3-K(r)$ is said to be {\it peripheral}
if it is homotopic to a loop on a peripheral torus.
A loop $\alpha_s$ is said to be {\it primitive}
if there is no element in the $2$-bridge link group $G(K(r))$
whose proper power is conjugate to $\alpha_s$.

\section{Technical lemmas}
\label{sec:technical_lemmas}

In this section, we set up Hypotheses~A, B and C,
under which we establish technical lemmas used for
the proofs in Sections~\ref{proof_for_twist_links_1}--\ref{proof_for_[2,1,n]}.

\subsection{Hypothesis~A}

\begin{hypothesisA}
\label{hyp:common}
{\rm
Let $r$ be a rational number such that $0 < r <1$ and
$r \neq 1/p$ for any integer $p \ge 2$.
For two distinct elements $s, s' \in I_1(r) \cup I_2(r)$,
suppose that the unoriented loops
$\alpha_s$ and $\alpha_{s'}$ are homotopic in $S^3-K(r)$.
Then $u_s$ and $u_{s'}^{\pm 1}$ are conjugate in $G(K(r))$.
Let $R$ be the symmetrized subset of $F(a, b)$
generated by the single relator $u_r$ of the upper presentation of $G(K(r))$,
and let $S(r)=(S_1, S_2, S_1, S_2)$ be the decomposition
as in \cite[Proposition~3.12]{lee_sakuma_2}.
Due to \cite[Lemma~4.7]{lee_sakuma_2},
there is a reduced annular $R$-diagram $M$ such that
$u_s$ and $u_{s'}^{\pm 1}$ are, respectively, outer and inner
boundary labels of $M$. Then we see
from \cite[Proposition~3.19(1)]{lee_sakuma_2}
that $M$ satisfies the three hypotheses (i), (ii) and (iii) of \cite[Theorem~4.9]{lee_sakuma_2}.

Let $J$ be the outer boundary layer of $M$ (see \cite[Figure~7(a)]{lee_sakuma_2}).
Also let $\alpha$ and $\delta$ be, respectively,
the outer and inner boundary cycles of $J$ starting from $v_0$,
where $v_0$ is a vertex lying in both the outer and inner boundaries of $J$.
Here, recall from \cite[Convention~4.6]{lee_sakuma_2} that
$\alpha$ is read clockwise and $\delta$ is read counterclockwise.
Let $\alpha=e_1, e_2, \dots, e_{2t}$ and
$\delta^{-1}=e_1', e_2', \dots, e_{2t}'$ be the
decompositions into oriented edges in $\partial J$.
Then clearly for each $i=1, \dots, t$,
there is a face $D_i$ of $J$ such that
$e_{2i-1}, e_{2i}, e_{2i}'^{-1}, e_{2i-1}'^{-1}$ are
consecutive edges in a boundary cycle of $D_i$.
We denote the path
$e_{2i-1}, e_{2i}$ by $\partial D_i^+$ and
the path $e_{2i-1}', e_{2i}'$
by $\partial D_i^-$.
In particular, if $J \subsetneq M$ (see \cite[Figure~7(b)]{lee_sakuma_2}),
then, for each $i=1, \dots, t$,
there is a face $D_i'$ in $M-J$ such that $e_{2i}'$ and $e_{2i+1}'$ are
two consecutive edges in $\partial D_i' \cap \delta^{-1}$.
Here the indices for the $2$-cells are considered modulo $t$,
and the indices for the edges are considered modulo $2t$.
}
\end{hypothesisA}

\begin{lemma}
\label{lem:vertex_position}
Under Hypothesis~A, both of the following hold
for every $i$.
\begin{enumerate}[\indent \rm (1)]
\item None of $S(\phi(e_{2i-1}))$, $S(\phi(e_{2i}))$,
$S(\phi(e_{2i}'))$ and $S(\phi(e_{2i-1}'))$ contains
$S_1$ as a subsequence.

\item None of $S(\phi(e_{2i-1}))$, $S(\phi(e_{2i}))$,
$S(\phi(e_{2i}'))$ and $S(\phi(e_{2i-1}'))$ contains
a subsequence of the form $(\ell, S_2, \ell')$,
where $\ell, \ell' \in \ZZ_+$.
\end{enumerate}
\end{lemma}

\begin{proof}
By \cite[Convention~4.3]{lee_sakuma_2},
each of the words $\phi(e_{2i-1})$, $\phi(e_{2i})$,
$\phi(e_{2i}')$ and $\phi(e_{2i-1}')$ is a piece
of the cyclic word $(u_r^{\pm 1})$.
So, the assertion follows from
(the only if part) of \cite[Corollary~3.25(1)]{lee_sakuma_2}.
\end{proof}

\begin{lemma}
\label{lem:two_cases}
Under Hypothesis~A, only one of the following holds
for each face $D_i$ of $J$.
\begin{enumerate}[\indent \rm (1)]
\item Both $S(\phi(\partial D_i^+))$ and $S(\phi(\partial D_i^-))$ contain
$S_1$ as their subsequence.

\item Both $S(\phi(\partial D_i^+))$ and $S(\phi(\partial D_i^-))$ contain
subsequences of the form $(\ell, S_2, \ell')$,
where $\ell, \ell' \in \ZZ_+$.
\end{enumerate}
\end{lemma}

\begin{proof}
By \cite[Convention~4.3]{lee_sakuma_2},
each of the words $\phi(\partial D_i^+)$ and $\phi(\partial D_i^-)$
is not a piece.
So, by (the if part of) \cite[Corollary~3.25(1)]{lee_sakuma_2},
each $S(\phi(\partial D_i^{\pm}))$ contains $S_1$ or
$(\ell, S_2, \ell')$ ($\ell, \ell' \in \ZZ_+$) as a subsequence.
On the other hand,
since $CS(\phi(\partial D_i^+)\phi(\partial D_i^-)^{-1})=CS(\phi(\partial D_i))=CS(u_r^{\pm 1})$
and since $S_1$ and $S_2$ are symmetric (\cite[Proposition~3.12(1)]{lee_sakuma_2}),
$CS(\phi(\partial D_i^+)\phi(\partial D_i^-)^{-1})$
is equal to $CS(u_r)=CS(r)=\lp S_1,S_2,S_1,S_2 \rp$.
So if $S(\phi(\partial D_i^{+}))$ contains $S_1$
(resp., $(\ell, S_2, \ell')$)
then $S(\phi(\partial D_i^{-}))$ cannot contain $(\ell, S_2, \ell')$
(resp., $S_1$).
Hence we obtain the desired result.
\end{proof}

\begin{lemma}
\label{lem:boundary_layer}
Under Hypothesis~A, only one of the following holds.
\begin{enumerate}[\indent \rm (1)]
\item For every face $D_i$ of $J$,
$S(\phi(\partial D_i^+))$ contains $S_1$ as its subsequence.

\item For every face $D_i$ of $J$,
$S(\phi(\partial D_i^+))$ contains a subsequence of the form $(\ell, S_2, \ell')$,
where $\ell, \ell' \in \ZZ_+$.
\end{enumerate}
\end{lemma}

\begin{proof} Suppose on the contrary that
(1) holds for $j$ and (2) holds for $j' \neq j$.
Then $CS(\phi(\alpha))=CS(u_s)=CS(s)$ contains both $S_1$ and $S_2$ as subsequences.
By \cite[Proposition~3.19(1)]{lee_sakuma_2}, $s \notin I_1(r) \cup I_2(r)$,
contradicting the hypothesis of the lemma.
\end{proof}

\subsection{Hypothesis~B} Assuming the following Hypothesis~B,
we will establish several technical lemmas concerning important properties
of $CS(\phi(\alpha))=CS(u_s)=CS(s)$.

\begin{hypothesisB}
\label{hyp:(1)holds}
{\rm
Suppose under Hypothesis~A that
Lemma~\ref{lem:boundary_layer}(1) holds,
namely, for every face $D_i$ of $J$,
suppose that $S(\phi(\partial D_i^+))$ contains
$S_1$ as its subsequence.
Then we can decompose the word $\phi(\alpha)$
(clearly $(u_s) \equiv (\phi(\alpha))$) into
\[
\phi(\alpha) \equiv y_1 w_1 z_1 y_2 w_2 z_2 \cdots y_t w_t z_t,
\]
where $\phi(\partial D_i^+) \equiv \phi(e_{2i-1}e_{2i}) \equiv y_iw_iz_i$,
$y_i$ and $z_i$ may be empty, $S(w_i)=S_1$, and where
$S(y_i w_i z_i)=(S(y_i), S_1, S(z_i))$
(here $S(y_i)$ and $S(z_i)$ are possibly empty), for every $i$.
By Lemma~\ref{lem:two_cases}, we also have the decomposition of the word $\phi(\delta^{-1})$
as follows (clearly $(u_{s'}^{\pm 1}) \equiv (\phi(\delta^{-1}))$ if $J=M$):
\[
\phi(\delta^{-1}) \equiv y_1' w_1' z_1' y_2' w_2' z_2' \cdots y_t' w_t' z_t',
\]
where $\phi(\partial D_i^-) \equiv \phi(e_{2i-1}'e_{2i}') \equiv y_i'w_i'z_i'$,
$y_i'$ and $z_i'$ may be empty, $S(w_i')=S_1$, and where
$S(y_i' w_i' z_i')=(S(y_i'), S_1, S(z_i'))$
(here $S(y_i')$ and $S(z_i')$ are possibly empty), for every $i$.
Then $S({y_i'}^{-1} y_i)=S(z_i z_i'^{-1})=S_2$ for every $i$.
}
\end{hypothesisB}

\begin{notation}
{\rm
Let $v$ be a reduced word in $\{a, b\}$.
If $v$ is not an empty word, then, by $v_b$ and $v_e$, respectively,
we denote a beginning subword and an ending subword of $v$
such that $|v_b|$ is the first term of the sequence $S(v)$
and $|v_e|$ is the last term of $S(v)$.
On the other hand, if $v$ is empty,
then $v_b$ and $v_e$ are also empty words.
(Though similar symbols, $v_{ib}$
and $v_{ie}$ ($1\le i\le 4$),
are used in different meanings
in \cite[Lemma~3.24]{lee_sakuma_2},
we believe this does not cause any confusion,
because these symbols are not used
in the remainder of this paper.)
}
\end{notation}

\begin{remark}
\label{rem:(1)holds}
{\rm
(1) If $r=[2, 2]=2/5$, then, by \cite[Example~3.17(2)]{lee_sakuma_2},
$CS(r)=\lp 3, 2, 3, 2 \rp$,
where $S_1=(3)$ and $S_2=(2)$.
So, in Hypothesis~B,
both $S(\phi(\partial D_i^+))$ and
$S(\phi(\partial D_i^-))$
are exactly of the form $(\ell, 3, \ell')$, where $0 \le \ell, \ell' \le 2$
are integers.

(2) If $r=[2, n]$ with $n \ge 3$, then, again by \cite[Example~3.17(2)]{lee_sakuma_2},
$CS(r)=\lp 3, (n-1) \langle 2 \rangle, 3, (n-1) \langle 2 \rangle \rp$,
where $S_1=(3)$ and $S_2=((n-1) \langle 2 \rangle)$.
So, in Hypothesis~B,
both $S(\phi(\partial D_i^+))$ and $S(\phi(\partial D_i^-))$
are exactly of the form
$(\ell_1, n_1 \langle 2 \rangle, 3, n_2 \langle 2 \rangle, \ell_2)$,
where $0 \le \ell_1, \ell_2 \le 1$ and $0 \le n_1, n_2 \le n-1$
are integers such that if $n_j=n-1$ then $\ell_j$ is necessarily $0$
for $j=1, 2$.
In particular, $S(y_{i,b})=(1)$ or $(2)$ unless $y_i$ is an empty word.
The same is true for $S(z_{i,e})$, $S(y_{i,b}')$ and $S(z_{i,e}')$.

(3) If $r=[2, 1, n]$ with $n \ge 2$, then, by \cite[Example~3.17(1)]{lee_sakuma_2},
$CS(r)=\lp n \langle 3 \rangle, 2, n \langle 3 \rangle, 2 \rp$,
where $S_1=(n \langle 3 \rangle)$ and $S_2=(2)$.
So, in Hypothesis~B,
both $S(\phi(\partial D_i^+))$ and
$S(\phi(\partial D_i^-))$
are exactly of the form $(\ell, n \langle 3 \rangle, \ell')$,
where $0 \le \ell, \ell' \le 2$ are integers.
In particular, $S(w_{i,b})=S(w_{i,e})=(3)$ and $S(w_{i,b}')=S(w_{i,e}')=(3)$.
}
\end{remark}

\begin{lemma}
\label{lem:case1-1(a)}
Let $r=n/(2n+1)=[2,n]$, where $n \ge 2$ is an integer.
Under Hypothesis~B,
suppose that $v$ is a subword of the cyclic word represented by
$\phi(\alpha) \equiv y_1 w_1 z_1 y_2 w_2 z_2 \cdots y_t w_t z_t$
such that $v$ corresponds to a
term of $CS(\phi(\alpha))=CS(s)$.
Then, after a cyclic shift of indices,
$v$ is equal to one of the following subwords:
\[
z_{0, e}w_1w_2\cdots w_qy_{q+1, b}, \quad
z_{0, e}w_1w_2\cdots w_q, \quad
w_1w_2\cdots w_qy_{q+1, b}, \quad
w_1w_2\cdots w_q,
\]
where $q\in \ZZ_+\cup\{0\}$ in the first three cases
and $q\in \ZZ_+$ in the last case.
In each of the above,
the ``intermediate subwords'' are empty;
to be precise, when we say that
$z_{0, e}w_1w_2\cdots w_qy_{q+1, b}$,
for example,
is a subword of $(u_s)$,
we assume that $y_1$, $z_iy_{i+1}$ $(1\le i\le q-1)$ and $z_q$
are empty words.
\end{lemma}

\begin{proof}
Recall from Hypothesis~B
that $S(y_i w_i z_i)=(S(y_i), S_1, S(z_i))$,
where $S(y_i)$ and $S(z_i)$ are possibly empty and $S_1=(3)$.
In other words, if $y_i$ (resp., $z_i$) is not an empty word,
then there is a sign change between $y_i$ and $w_i$
(resp., between $w_i$ and $z_i$).
The desired result follows from this observation.
\end{proof}

Throughout the remainder of this paper,
we will assume the following convention.

\begin{convention}
\label{con:figure}
{\rm
In Figures~\ref{fig.lemma-1-1-i}--\ref{fig.III-case-2b},
the change of directions of consecutive arrowheads
represents the change from positive (negative, resp.) words
to negative (positive, resp.) words, and
a dot represents a vertex whose position is clearly identified.
Also an Arabic number represents the length of the corresponding positive
(or negative) word.
In Figures~\ref{fig.lemma-1-1-i}--\ref{fig.III-lemma-2-3},
the upper complementary region is regarded as the unbounded region
of $\RR^2-M$.
Thus the outer boundary cycles runs
the upper boundary from left to right.
}
\end{convention}

The following Lemmas~\ref{lem:case1-1(b)} and \ref{lem:case2-1(b)} show
that there are strong restrictions
for the shape of the word $v$ in Lemma~\ref{lem:case1-1(a)}.

\begin{lemma}
\label{lem:case1-1(b)}
Let $r=2/5=[2,2]$. Under Hypothesis~B,
the following hold for every $i$.
\begin{enumerate}[\indent \rm (1)]
\item $S(z_{i}y_{i+1})\ne (3)$.

\item $S(w_iz_iy_{i+1})\ne (4)$ and $S(z_iy_{i+1}w_{i+1})\ne (4)$.

\item If $J \subsetneq M$, then $S(w_iz_iy_{i+1}w_{i+1})\ne (6)$.

\item $S(z_{i-1}y_iw_iz_iy_{i+1})\ne (7)$.

\item If $J=M$, then
$S(z_{i-1}y_iw_iz_iy_{i+1}w_{i+1})\ne (8)$ and
$S(w_{i-1}z_{i-1}y_iw_iz_iy_{i+1})\ne (8)$.

\item $S(w_iz_iy_{i+1}w_{i+1})\ne (3,3)$.
\end{enumerate}
\end{lemma}

\begin{proof}
(1) Suppose on the contrary that $S(z_iy_{i+1})=(3)$ for some $i$.
Without loss of generality, we may assume $S(z_1y_2)=(3)$.
Then, since $0\le |z_1|, |y_2| \le 2$,
we have either $(|z_1|, |y_2|)=(1, 2)$
or $(|z_1|, |y_2|)=(2, 1)$.
We now assume $(|z_1|, |y_2|)=(2, 1)$.
(The other case can be treated similarly.)
Then by using the fact that $CS(\phi(\partial D_i))=CS(2/5)=\lp 3,2,3,2 \rp$,
we see that
$J$ is locally as illustrated in Figure~\ref{fig.lemma-1-1-i}(a)
which follows Convention~\ref{con:figure}.
The Arabic numbers
$3$, $2$, $1$ and $3$ near the upper boundary
represent the lengths of the words $w_1$, $z_1$, $y_2$ and $w_2$,
respectively,
whereas the Arabic numbers
$3$, $1$ and $3$ near the lower boundary
represent the lengths of the words $w_1'$, $y_2'$ and $w_2'$
respectively (in particular $|z_1'|=0$),
and the change of directions of consecutive arrowheads
represents the change from positive (negative, resp.) words
to negative (positive, resp.) words.

Suppose first that $J=M$.
Then we see from
Figure~\ref{fig.lemma-1-1-i}(a) that
$CS(\phi(\delta^{-1}))=CS(u_{s'}^{\pm 1})=CS(s')$
involves both a term $1$ and a term of the form $3+c$
with $c \in \ZZ_+ \cup \{0\}$, contradicting \cite[Lemma~3.8]{lee_sakuma_2}
which says that either $CS(s')$ is equal to $\lp m,m \rp$ or $CS(s')$
consists of $m$ and $m+1$ for some $m\in\ZZ_+$.
Suppose next that $J \subsetneq M$.
Then by Lemma~\ref{lem:vertex_position}(1),
none of $S(\phi(e_{1}'))$ and $S(\phi(e_{2}'))$ contains
$S_1=(3)$ as a subsequence.
This means that
the initial vertex of $e_2'$ lies in the interior
of the segment of $\partial D_1^-$ with weight $3$.
(See Figure~\ref{fig.lemma-1-1-i}(b),
where the initial vertex of $e_2'$ is the left-most vertex.)
Similarly, the terminal vertex of $e_3'$
lies in the interior
of the segment of $\partial D_2^-$ with weight $3$;
in particular, it does not lie in
the segment of $\partial D_2^-$ with weight $1$.
Hence, we see from Figure~\ref{fig.lemma-1-1-i}(b)
that $S(\phi(e_2'e_3'))$ is of the form $(\ell_1, 1, \ell_2)$
with $\ell_1, \ell_2 \in \ZZ_+$.
This yields that a term $1$
occurs in $CS(\phi(\partial D_1'))=CS(2/5)=\lp 3, 2, 3, 2\rp$,
which is obviously a contradiction.

\begin{figure}[h]
\includegraphics{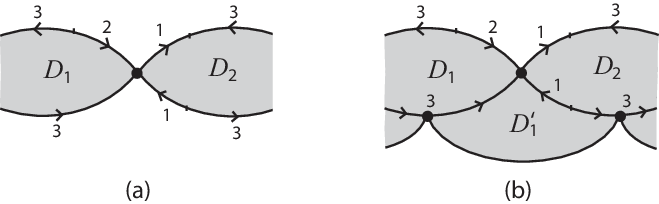}
\caption{
Lemma~\ref{lem:case1-1(b)}(1) where $S(z_1y_2)=(2+1)$}
\label{fig.lemma-1-1-i}
\end{figure}

(2) Suppose on the contrary that $S(w_1z_1y_2)=(4)$.
(The other case is treated similarly.)
Then, since $w_1$ and $z_1$ have different signs when $z_1$ is non-empty
and since $|w_1|=3$ and $0\le |z_1|, |y_2| \le 2$,
the only possibility is that $|z_1|=0$ and $S(w_1y_2)=(4)$.
If $J=M$,
then we see from Figure~\ref{fig.lemma-1-2}(a) that
$CS(s')$ involves both a term $1$ and a term of the form $3+c$
with $c \in \ZZ_+ \cup \{0\}$, contradicting \cite[Lemma~3.8]{lee_sakuma_2}.
On the other hand, if $J \subsetneq M$,
then we see, by using Lemma~\ref{lem:vertex_position}(1)
as in the proof of Lemma~\ref{lem:case1-1(b)}(1),
that $S(\phi(e_2'e_3'))$ is of the form
$(\ell_1, 2, 1, \ell_2)$
with $\ell_1, \ell_2 \in \ZZ_+$
(see Figure~\ref{fig.lemma-1-2}(b)).
This implies that $CS(\phi(\partial D_1'))=CS(2/5)$ has a term $1$,
a contradiction.

\begin{figure}[h]
\includegraphics{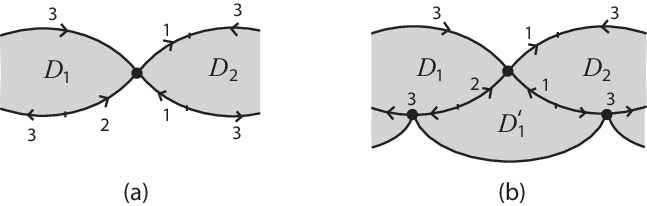}
\caption{
Lemma~\ref{lem:case1-1(b)}(2) where $S(w_1z_1y_2)=(3+0+1)$}
\label{fig.lemma-1-2}
\end{figure}

(3) Suppose  $J \subsetneq M$ and suppose on the contrary that
$S(w_0z_0y_{1}w_{1})= (6)$.
Then $|z_0|=|y_1|=0$ and
we see, by using Lemma~\ref{lem:vertex_position}(1)
as in the proof of Lemma~\ref{lem:case1-1(b)}(1),
that
$S(\phi(e_2'e_3'))$ is of the form $(\ell_1, 4, \ell_2)$
with $\ell_1, \ell_2 \in \ZZ_+$, as illustrated in
Figure~\ref{fig.lemma-1-4b(a)}.
This implies that $CS(\phi(\partial D_1'))=CS(2/5)$ contains a term $4$,
a contradiction.

\begin{figure}[h]
\includegraphics{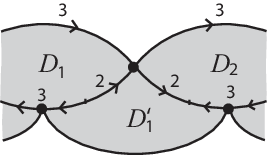}
\caption{
Lemma~\ref{lem:case1-1(b)}(3)
where $S(w_0z_0y_{1}w_{1})=(3+0+0+3)$
}
\label{fig.lemma-1-4b(a)}
\end{figure}

(4) Suppose on the contrary that
$S(z_{0}y_1w_1z_1y_{2})= (7)$.
Then we have $|z_0|=|y_2|=2$ and $|y_1|=|z_1|=0$.
If $J=M$, then
we see from Figure~\ref{fig.lemma-1-5}(a) that
$CS(s')$ includes both a term $3$ and a term of the form
$5+c$ with $c \in \ZZ_+ \cup \{0\}$,
contradicting \cite[Lemma~3.8]{lee_sakuma_2}.
So, we must have $J \subsetneq M$.
Note that $S(\phi(\partial D_1^-))=(2,3,2)$
and that the terminal vertex of $e_1'$ lies in the interior
of the segment of $\partial D_1^-$ with weight $3$,
by Lemma~\ref{lem:vertex_position}(1).
We may assume the vertex divides the segment
into segments with weights $1$ and $2$
as in Figure~\ref{fig.lemma-1-5}(b).
(The other case where the weights of $1$ and $2$ are interchanged
is treated similarly.)
Then, as illustrated in Figure~\ref{fig.lemma-1-5}(b),
$CS(\phi(\delta_1^{-1}))$ includes a subsequence $(3, 1, 3)$,
where $\delta_1$ is an inner boundary cycle of
the outer boundary layer, say $J_1$, of $M-J$.
If $M=J \cup J_1$, then $CS(\phi(\delta_1^{-1}))=CS(s')$ contains
both a term $1$ and a term $3$, contradicting
\cite[Lemma~3.8]{lee_sakuma_2}.
On the other hand, if $M \supsetneq J \cup J_1$,
then,
by using the same argument as in the case $M \supsetneq J$ of (1)
replacing the consideration of $2$-cells $D_1$ and $D_2$ in Figure~\ref{fig.lemma-1-1-i}(b)
with the consideration of
$2$-cells $D_0'$ and $D_1'$ in Figure~\ref{fig.lemma-1-5}(b),
respectively, we obtain a contradiction.

\begin{figure}[h]
\includegraphics{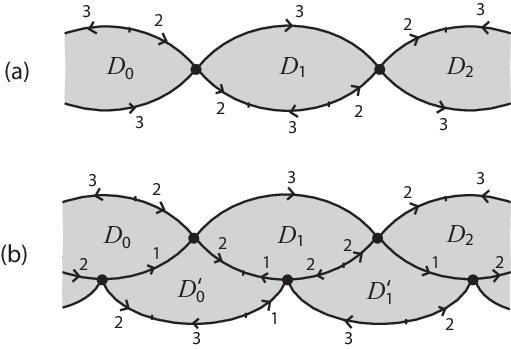}
\caption{
Lemma~\ref{lem:case1-1(b)}(4)
where $S(z_{0}y_1w_1z_1y_{2})=(2+0+3+0+2)$}
\label{fig.lemma-1-5}
\end{figure}

(5) Suppose on the contrary that $J=M$ and
$S(z_0y_1w_1z_1y_2w_2)= (8)$.
(The other case is treated similarly.)
Then $|z_0|=2$, $|y_1|=|z_1|=|y_2|=0$,
and we see from Figure~\ref{fig.lemma-1-6} that
$CS(s')$ includes both a term $3$ and a term of the form $5+c$
with $c \in \ZZ_+ \cup \{0\}$, contradicting \cite[Lemma~3.8]{lee_sakuma_2}.

\begin{figure}[h]
\includegraphics{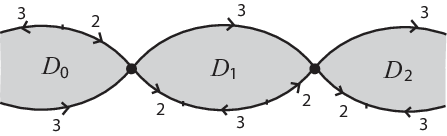}
\caption{
Lemma~\ref{lem:case1-1(b)}(5)
where $S(z_0y_1w_1z_1y_2w_2)=(2+0+3+0+0+3)$}
\label{fig.lemma-1-6}
\end{figure}

(6) Suppose on the contrary that
$S(w_1z_1y_{2}w_{2})= (3,3)$.
Then $|z_1|=|y_2|=0$.
If $J=M$ (see Figure~\ref{fig.lemma-1-1-ii}(a)),
then $CS(s')$ contains both a term $2$ and a term of the form $3+c$
with $c \in \ZZ_+ \cup \{0\}$.
Here, if $c=0$, then $s' \notin I_1(2/5) \cup I_2(2/5)$ by \cite[Proposition~3.19(1)]{lee_sakuma_2},
contradicting the hypothesis of the theorem,
while if $c>0$, then we have a contradiction to \cite[Lemma~3.8]{lee_sakuma_2}.
On the other hand, if $J \subsetneq M$ (see Figure~\ref{fig.lemma-1-1-ii}(b)),
then we see,
by using Lemma~\ref{lem:vertex_position}(1)
as in the proof of Lemma~\ref{lem:case1-1(b)}(1), that
$S(\phi(e_2'e_3'))$ is of the form $(\ell_1, 2, 2, \ell_2)$
with $\ell_1, \ell_2 \in \ZZ_+$.
This implies that a subsequence $(2, 2)$ occurs
in $CS(\phi(\partial D_1'))=CS(2/5)$, which is a contradiction.
\end{proof}

\begin{figure}[h]
\includegraphics{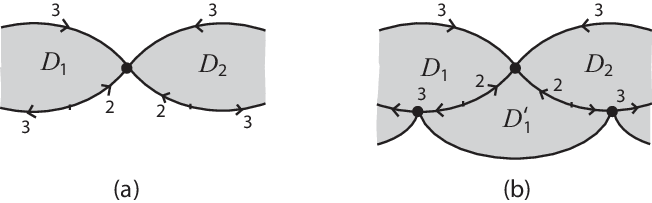}
\caption{
Lemma~\ref{lem:case1-1(b)}(6) where
$S(w_1z_1y_{2}w_{2})=(3,3)$}
\label{fig.lemma-1-1-ii}
\end{figure}

\begin{lemma}
\label{lem:case1-1}
Let $r=2/5=[2,2]$. Under Hypothesis~B,
the following hold, where $d \in \ZZ_+ \cup \{0\}$.
\begin{enumerate}
\item[\indent \rm (1)] No two consecutive terms of $CS(s)$ can be $(3, 3)$.

\item[\rm (2)] No two consecutive terms of $CS(s)$ can be $(4, 4)$.

\item[\rm (3a)] No two consecutive terms of $CS(s)$ can be of the form $(6+3d, 6+3d)$.

\item[\rm (3b)] If $J \subsetneq M$, then
no term of $CS(s)$ can be of the form $6+d$.

\item[\rm (4)] No term of $CS(s)$ can be of the form $7+3d$.

\item[\rm (5)] No term of $CS(s)$ can be of the form $8+3d$.
\end{enumerate}
\end{lemma}

\begin{proof}
(1) Suppose on the contrary that $CS(\phi(\alpha))=CS(s)$ contains $(3, 3)$
as a subsequence.
Let $v=v'v''$ be a subword of the cyclic word $(u_s)$
corresponding to a subsequence $(3, 3)$,
where $S(v')=S(v'')=(3)$.
By using Lemma~\ref{lem:case1-1(a)}
and the facts that $0 \le |z_i|, |y_i| \le 2$ and $|w_i|=3$,
we see that one
of the following holds after a shift of indices.
\begin{enumerate}[\indent \rm (i)]
\item $(v',v'')=(z_1y_2,w_2)$, where $S(z_1y_2)=(3)$.

\item $(v',v'')=(w_1,z_1y_2)$, where $S(z_1y_2)=(3)$.

\item $(v',v'')=(w_1,w_2)$.
\end{enumerate}
However, (i) and (ii) are impossible
by Lemma~\ref{lem:case1-1(b)}(1),
and (iii) is impossible
by Lemma~\ref{lem:case1-1(b)}(6).

(2) Suppose on the contrary that $CS(s)$
contains $(4, 4)$ as a subsequence.
Let $v=v'v''$ be a subword of the cyclic word $(u_s)$
corresponding to a subsequence $(4, 4)$,
where $S(v')=S(v'')=(4)$.
If $v'$ contains $w_i$ for some $i$,
then we see,
by using Lemma~\ref{lem:case1-1(a)}
and the identity $|w_i|=3$, that
either $v'=w_iz_iy_{i+1}$ with
$(|z_i|,|y_{i+1}|)=(0,1)$
or
$v'=z_{i-1}y_{i}w_{i}$ with
$(|z_{i-1}|,|y_{i}|)=(1,0)$.
However both cases are impossible by
Lemma~\ref{lem:case1-1(b)}(2).
Thus $v'$ cannot contain $w_i$.
Since $S(v')=(4)$ is a term of $CS(s)$,
this implies that $v'$ is disjoint from $w_i$
for every $i$.
The same conclusion also holds for $v''$,
and hence for $v=v'v''$.
Thus $v$ is a subword of $z_iy_{i+1}$ for some $i$.
This is a contradiction,
because $|v|=8$ whereas $|z_iy_{i+1}|\le 4$.

(3a) Suppose on the contrary that $CS(s)$
contains $(6+3d, 6+3d)$ as a subsequence.
Let $v=v'v''$ be a subword of the cyclic word $(u_s)$
corresponding to a subsequence $(6+3d, 6+3d)$,
where $S(v')=S(v'')=(6+3d)$.
By using Lemma~\ref{lem:case1-1(a)}
and the facts that $0\le |y_i|, |z_i| \le 2$ and $|w_i|=3$,
we see that one of the following holds
after a cyclic shift of indices.
\begin{enumerate}[\indent \rm (i)]
\item $v'=w_1w_2\cdots w_q$ with $q=d+2$.

\item $v'=z_0w_1w_2\cdots w_qy_{q+1}$ with $q=d+1$,
where $(|z_0|,|y_{q+1}|)=(1,2)$ or $(2,1)$.
\end{enumerate}
If (ii) holds, then either $S(z_0y_1w_1)=(4)$ or
$S(w_qz_qy_{q+1})=(4)$,
a contradiction to Lemma~\ref{lem:case1-1(b)}(2).
So (i) holds.
By applying the same argument to $v''$
and by using the fact that $v'v''$ is a subword of $(u_s)$,
we see that
$v''=w_{q+1}w_{q+2}\cdots w_{2q}$,
where $z_qy_{q+1}$ is empty.
Hence we see
$S(w_qz_qy_{q+1}w_{q+1})=S(w_qw_{q+1})=(3,3)$.
This contradicts Lemma~\ref{lem:case1-1(b)}(6).

(3b) Suppose $J \subsetneq M$ and suppose on the contrary that
$CS(s)$ contains a term $6+d$.
First suppose that $CS(s)$ contains a term $6$.
Let $v$ be a subword of the cyclic word $(u_s)$
corresponding to a term $6$.
By arguing as in the proof of Lemma~\ref{lem:case1-1}(3a),
we see that one of the following holds.
\begin{enumerate}[\indent \rm (i)]
\item $v=w_1w_2$.

\item $v=z_0w_1y_{2}$,
where $(|z_0|,|y_{q+1}|)=(1,2)$ or $(2,1)$.
\end{enumerate}
However, (i) is impossible by Lemma~\ref{lem:case1-1(b)}(3),
and (ii) is impossible by Lemma~\ref{lem:case1-1(b)}(2),
as in the proof of Lemma~\ref{lem:case1-1}(3a).

Next suppose that $CS(s)$ involves a term $7$.
Let $v$ be a subword of the cyclic word $(u_s)$
corresponding to a term $7$.
Arguing as in the proof of Lemma~\ref{lem:case1-1}(3a),
we may assume that one of the following holds.
\begin{enumerate}[\indent \rm (i)]
\item $v=z_0w_1w_2$ with $|z_0|=1$.

\item $v=w_1w_2y_3$ with $|y_3|=1$.

\item $v=z_0w_1y_{2}$,
where $|z_0|=|y_2|=2$.
\end{enumerate}
However, (i) and (ii) are impossible by Lemma~\ref{lem:case1-1(b)}(2),
and (iii) is impossible by Lemma~\ref{lem:case1-1(b)}(4).

Finally suppose that $CS(s)$ contains a term of the form $8+d$.
Let $v$ be a subword of the cyclic word $(u_s)$
corresponding to a term $8+d$.
By using Lemma~\ref{lem:case1-1(a)}
and the facts that $0\le |y_i|, |z_i| \le 2$ and $|w_i|=3$,
we see that $v$ must contain a subword $w_iw_{i+1}$ for some $i$.
This contradicts Lemma~\ref{lem:case1-1(b)}(3).

(4) By Lemma~\ref{lem:case1-1}(3b),
it remains to prove the assertion for $J=M$.
Suppose $J=M$ and
suppose on the contrary that $CS(s)$ contains a term
of the form $7+3d$.
Let $v$ be a subword of the cyclic word $(u_s)$
corresponding to a term $7+3d$.
Arguing as in the proof of Lemma~\ref{lem:case1-1}(3a),
we may assume that one of the following holds.
\begin{enumerate}[\indent \rm (i)]
\item $v=z_0w_1w_2\cdots w_q$ where $|z_0|=1$ and $q=2+d$.

\item $v=w_1w_2\cdots w_q y_{q+1}$ where $|y_{q+1}|=1$ and $q=2+d$.

\item $v=z_0w_1w_2\cdots w_q y_{q+1}$,
where $|z_0|=|y_{q+1}|=2$ and $q=1+d$.
\end{enumerate}
However, (i) and (ii) are impossible by
Lemma~\ref{lem:case1-1(b)}(2), and
(iii) is impossible by Lemma~\ref{lem:case1-1(b)}(4) and (5).

(5) By Lemma~\ref{lem:case1-1}(3b),
it remains to prove the assertion for $J=M$.
Suppose $J=M$ and
suppose on the contrary that $CS(s)$ contains a term
of the form $8+3d$.
Let $v$ be a subword of the cyclic word $(u_s)$
corresponding to a term $8+3d$.
Arguing as in the proof of Lemma~\ref{lem:case1-1}(3a),
we may assume that one of the following holds.
\begin{enumerate}[\indent \rm (i)]
\item $v=z_0w_1w_2\cdots w_q$ where $|z_0|=2$ and $q=2+d$.

\item $v=w_1w_2\cdots w_q y_{q+1}$ where $|y_{q+1}|=2$ and $q=2+d$.

\item $v=z_0w_1w_2\cdots w_q y_{q+1}$,
where $|z_0|=|y_{q+1}|=1$ and $q=2+d$.
\end{enumerate}
However, (i) and (ii) are impossible by Lemma~\ref{lem:case1-1(b)}(5),
and (iii) is impossible by Lemma~\ref{lem:case1-1(b)}(2).
\end{proof}

\begin{corollary}
\label{cor:case1-1}
Let $r=2/5=[2,2]$. Under Hypothesis~B,
$CS(s)$ satisfies one of the following conditions.
\begin{enumerate}[\indent \rm (1)]
\item $CS(s)=\lp 5,5 \rp$.

\item $CS(s)$ has the form consisting of $m$ and $m+1$,
where $m$ is one of $2, 3, 4$ and $5$.
\end{enumerate}
\end{corollary}

\begin{proof}
By \cite[Lemma~3.8]{lee_sakuma_2},
either $CS(s)=\lp m,m \rp$ or $CS(s)$ consists of $m$ and $m+1$
with $m\in \ZZ_+$.
By Hypothesis~B together with
Remark~\ref{rem:(1)holds}(1), $\phi(\alpha)$ involves
a subword $w_i$ whose $S$-sequence is $(3)$,
so $CS(\phi(\alpha))=CS(u_s)=CS(s)$ must contain
a term of the form $3+c$, where $c \in \ZZ_+ \cup \{0\}$.
If $CS(s)=\lp m,m \rp$, then $m\ge 3$ by this observation.
Moreover, by Lemma~\ref{lem:case1-1},
$m$ is not equal to $3, 4, 6+3d, 7+3d$ nor $8+3d$
for any $d \in \ZZ_+ \cup \{0\}$.
Hence $m=5$.
On the other hand, if $m$ consists of $m$ and $m+1$,
then by Lemma~\ref{lem:case1-1},
none of $m$ and $m+1$ is equal to
$7+3d$ nor $8+3d$ for any $d \in \ZZ_+ \cup \{0\}$.
Thus $m$ is less than $5$.
Since $CS(s)$ involves a term $3+c$,
we have $m+1\ge 3$.
Hence we see $2\le m\le 5$.
\end{proof}

Next, we study the case where
$r=n/(2n+1)=[2,n]$ with $n \ge 3$.
Recall from Remark~\ref{rem:(1)holds}(2) that
$CS(r)=\lp 3, (n-1) \langle 2 \rangle, 3, (n-1) \langle 2 \rangle \rp$,
where $S_1=(3)$ and $S_2=((n-1) \langle 2 \rangle)$.
Recall also that $S(y_{i,b})=(1)$ or $(2)$ unless $y_i$ is an empty word,
and that $S(z_{i,e})=(1)$ or $(2)$ unless $z_i$ is an empty word,
for every $i$.
The following lemma is a counterpart of
Lemma~\ref{lem:case1-1(b)}.

\begin{lemma}
\label{lem:case2-1(b)}
Let $r=n/(2n+1)=[2,n]$, where $n \ge 3$ is an integer.
Under Hypothesis~B,
the following hold for every $i$.
\begin{enumerate}[\indent \rm (1)]
\item $S(z_{i, e}y_{i+1, b})\ne (3)$.

\item $S(w_iz_iy_{i+1, b})\ne (4)$ and $S(z_{i, e}y_{i+1}w_{i+1})\ne (4)$.

\item $S(w_iz_iy_{i+1, b})\ne (5)$ and $S(z_{i, e}y_{i+1}w_{i+1})\ne (5)$.

\item $S(w_iz_iy_{i+1}w_{i+1})\ne (6)$.

\item $S(w_iz_iy_{i+1}w_{i+1})\ne (3,3)$.
\end{enumerate}
\end{lemma}

\begin{proof}
The proofs of (1), (2) and (5), respectively,
are parallel to those of (1), (2) and (6)
in Lemma~\ref{lem:case1-1(b)}.
We only have to replace the subsequence $(2)$ with
the sequence $((n-1) \langle 2 \rangle)$.

(3) Suppose on the contrary that $S(z_{1, e}y_2w_2)= (5)$.
Then $|z_{1, e}|=2$ and $|y_2|=0$, and $J$ is as depicted in
Figure~\ref{fig.II-lemma-1-4}(a).
If $J=M$,
then we see from Figure~\ref{fig.II-lemma-1-4}(a) that
$CS(\phi(\delta^{-1}))=CS(s')$ includes both a term $2$ and a term $4$,
contradicting \cite[Lemma~3.8]{lee_sakuma_2}.
If $J \subsetneq M$,
then by using Lemma~\ref{lem:vertex_position}(1)
as in the proof of Lemma~\ref{lem:case1-1(b)}(1),
we can see that $S(\phi(e_2'e_3'))$
contains $(\ell_1, 4, \ell_2)$ as a subsequence,
where $\ell_1, \ell_2 \in \ZZ_+$.
This implies that a term $4$ occurs in $CS(r)$, a contradiction.

(4) Suppose on the contrary that $S(w_1z_1y_2w_2)=(6)$.
Then $|z_1|=|y_2|=0$ and $J$ is as depicted in
Figure~\ref{fig.II-lemma-1-4}(b).
We obtain a contradiction as in the proof of
Lemma~\ref{lem:case2-1(b)}(3).
\end{proof}

\begin{figure}[h]
\includegraphics{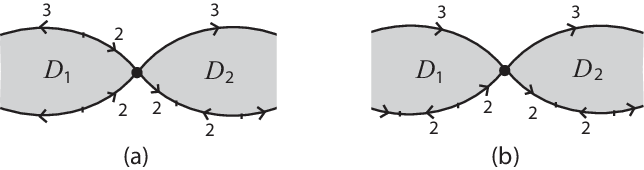}
\caption{
(a) Lemma~\ref{lem:case2-1(b)}(3)
where $S(z_1y_2w_2)=(2+0+3)$, and
(b) Lemma~\ref{lem:case2-1(b)}(4)
where $S(w_1z_1y_2w_2)=(3+0+0+3)$}
\label{fig.II-lemma-1-4}
\end{figure}

\begin{lemma}
\label{lem:case2-1}
Let $r=n/(2n+1)=[2,n]$, where $n \ge 3$ is an integer.
Under Hypothesis~B, the following hold.
\begin{enumerate}[\indent \rm (1)]
\item No two consecutive terms of $CS(s)$ can be $(3, 3)$.

\item No two consecutive terms of $CS(s)$ can be $(4, 4)$.

\item No term of $CS(s)$ can be of the form $5+d$, where $d \in \ZZ_+ \cup \{0\}$.
\end{enumerate}
\end{lemma}

\begin{proof}
The proofs of (1) and (2)
are parallel to those of (1) and (2)
in Lemma~\ref{lem:case1-1},
where we have only to use
Lemma~\ref{lem:case2-1(b)}(1), (2) and (5)
instead of
Lemma~\ref{lem:case1-1(b)}(1), (2) and (6).

(3) Suppose on the contrary that $CS(s)$ has a term of the form $5+d$.
Let $v$ be a subword of the cyclic word $(u_s)$
corresponding to a term $5+d$.
By using Lemma~\ref{lem:case1-1(a)}
and Remark~\ref{rem:(1)holds}(2),
we see that one of the following holds
after a cyclic shift of indices.
\begin{enumerate}[\indent \rm (i)]
\item $v$ contains $z_{0, e}w_1$ with $S(z_{0, e}w_1)=(4)$ or $(5)$.

\item $v$ contains $w_1y_{2, b}$ with $S(w_1y_{2, b})=(4)$ or $(5)$.

\item $v$ contains $w_1w_2$ with $S(w_1w_2)=(6)$.
\end{enumerate}
However, (i) and (ii) are impossible by
Lemma~\ref{lem:case2-1(b)}(2) and (3),
and (iii) is impossible by Lemma~\ref{lem:case2-1(b)}(4).
\end{proof}

Finally, we study the case where
$r=(n+1)/(3n+2)=[2,1,n]$ with $n \ge 2$.
Recall from Remark~\ref{rem:(1)holds}(3)
that
$CS(r)=\lp n \langle 3 \rangle, 2, n \langle 3 \rangle, 2 \rp$,
where $S_1=(n \langle 3 \rangle)$ and $S_2=(2)$.
Recall also
$S(w_{i, b})=S(w_{i, e})=(3)$ for every $i$.
The following lemma is a counterpart of
Lemmas~\ref{lem:case1-1(b)} and \ref{lem:case2-1(b)}.

\begin{lemma}
\label{lem:case3-1(b)}
Let $r=(n+1)/(3n+2)=[2,1,n]$, where $n \ge 2$ is an integer.
Under Hypothesis~B, the following hold for every $i$.
\begin{enumerate}[\indent \rm (1)]
\item $S(z_{i}y_{i+1})\ne (3)$.

\item $S(w_{i,e}z_iy_{i+1})\ne (4)$ and
$S(z_iy_{i+1}w_{i+1,b})\ne (4)$.

\item $S(w_{i,e}z_iy_{i+1})\ne (5)$
and
$S(z_iy_{i+1}w_{i+1,b})\ne (5)$.

\item $S(w_{i,e}z_iy_{i+1}w_{i+1,b})\ne (3,3)$.
\end{enumerate}
\end{lemma}

\begin{proof}
The proofs of (1), (2) and (4) are
parallel to those of Lemma~\ref{lem:case1-1(b)}(1), (2) and (6),
respectively.
We only have to replace the term $3$ with
the sequence $(n \langle 3 \rangle)$.

(3) Suppose on the contrary
$S(w_{1,e}z_1y_2)= (5)$.
(The other case is treated similarly.)
Then $|z_1|=0$ and $|y_2|=2$.
Thus $S(y_2w_2)=(2, n \langle 3 \rangle)$, and so
$CS(s)$ has a term $3$. (Here we use the assumption $n\ge 2$.)
This contradicts \cite[Lemma~3.8]{lee_sakuma_2}
because $CS(s)$ contains a term $5+d$ with $d\ge 0$
by the assumption.
\end{proof}

\begin{lemma}
\label{lem:case3-1}
Let $r=(n+1)/(3n+2)=[2,1,n]$, where $n \ge 2$ is an integer.
Under Hypothesis~B, the following hold.

\begin{enumerate}[\indent \rm (1)]
\item No two consecutive terms of $CS(s)$ can be $(4, 4)$.

\item No term of $CS(s)$ can be $5$.

\item No term of $CS(s)$ can be of the form $7+d$, where $d \in \ZZ_+ \cup \{0\}$.
\end{enumerate}
\end{lemma}

\begin{proof}
(1) The proof is parallel to that of
Lemma~\ref{lem:case1-1}(2),
where we have only to use Lemma~\ref{lem:case3-1(b)}(2)
instead of Lemma~\ref{lem:case1-1(b)}(2).

(2) Suppose on the contrary that $5$ occurs in $CS(s)$.
Let $v$ be a subword of the cyclic word $(u_s)$
corresponding to a term $5$.
Then, by using Lemma~\ref{lem:case1-1(a)}
and Remark~\ref{rem:(1)holds}(3),
we see that one of the following holds
after a shift of indices.
\begin{enumerate}[\indent \rm (i)]
\item $v=z_0w_{1,b}$ with $|z_0|=2$.

\item $v=w_{1, e}y_2$ with $|y_2|=2$.
\end{enumerate}
However, this contradicts Lemma~\ref{lem:case3-1(b)}(3).

(3) Note that every term of $CS(s)$
is at most $6$, and that this happens only when
$S(w_iw_{i+1})
=((n-1) \langle 3 \rangle, 6, (n-1) \langle 3 \rangle)$,
where $|z_i|=|y_{i+1}|=0$, for some $i$.
Hence we obtain the desired result.
\end{proof}

\subsection{Hypothesis~C} Assuming the following Hypothesis~C,
we will establish three technical lemmas (Lemmas~\ref{lem:case1-2}--\ref{lem:case3-2})
concerning the sequence $S(z_iy_{i+1})$
accordingly as $r=[2, 2]$, $r=[2, n]$ with $n \ge 3$,
and $r=[2, 1, n]$ with $n \ge 2$.

\begin{hypothesisC}
\label{hyp:(2)holds}
{\rm
Suppose under Hypothesis~A that
Lemma~\ref{lem:boundary_layer}(2) holds,
namely, for every face $D_i$ of $J$,
suppose that $S(\phi(\partial D_i^+))$ contains
a subsequence of the form $(\ell, S_2, \ell')$,
where $\ell, \ell' \in \ZZ_+$.
Then we can decompose the word $\phi(\alpha)$
(clearly $(u_s) \equiv (\phi(\alpha))$) into
\[
\phi(\alpha) \equiv y_1 w_1 z_1 y_2 w_2 z_2 \cdots y_t w_t z_t,
\]
where $\phi(\partial D_i^+) \equiv \phi(e_{2i-1}e_{2i}) \equiv y_iw_iz_i$,
$y_i$ and $z_i$ are nonempty words, $S(w_i)=S_2$,
and where $S(y_iw_iz_i)=(S(y_i), S_2, S(z_i))$, for every $i$.
By Lemma~\ref{lem:two_cases}, we also have the decomposition of the word $\phi(\delta^{-1})$
as follows (clearly $(u_{s'}^{\pm 1}) \equiv (\phi(\delta^{-1}))$ if $J=M$):
\[
\phi(\delta^{-1}) \equiv y_1' w_1' z_1' y_2' w_2' z_2' \cdots y_t' w_t' z_t',
\]
where $\phi(\partial D_i^-) \equiv \phi(e_{2i-1}'e_{2i}') \equiv y_i'w_i'z_i'$,
$y_i'$ and $z_i'$ are nonempty words, $S(w_i')=S_2$,
and where $S(y_i'w_i'z_i')=(S(y_i'), S_2, S(z_i'))$, for every $i$.
Then $S(y_i'^{-1} y_i)=S(z_i z_i'^{-1})=S_1$ for every $i$.
}
\end{hypothesisC}

\begin{remark}
\label{rem:(2)holds}
{\rm
(1) If $r=[2, 2]=2/5$, then $CS(r)=\lp 3, 2, 3, 2 \rp$,
where $S_1=(3)$ and $S_2=(2)$.
So, in Hypothesis~C,
both $S(\phi(\partial D_i^+))$ and $S(\phi(\partial D_i^-))$
are exactly of the form $(\ell, 2, \ell')$, where $1 \le \ell, \ell' \le 2$
are integers.

(2) If $r=[2, n]$ with $n \ge 3$, then
$CS(r)=\lp 3, (n-1) \langle 2 \rangle, 3, (n-1) \langle 2 \rangle \rp$,
where $S_1=(3)$ and $S_2=((n-1) \langle 2 \rangle)$.
So, in Hypothesis~C,
both $S(\phi(\partial D_i^+))$ and $S(\phi(\partial D_i^-))$
are exactly of the form $(\ell, (n-1) \langle 2 \rangle, \ell')$,
where $1 \le \ell, \ell' \le 2$ are integers.

(3) If $r=[2, 1, n]$ with $n \ge 2$, then
$CS(r)=\lp n \langle 3 \rangle, 2, n \langle 3 \rangle, 2 \rp$,
where $S_1=(n \langle 3 \rangle)$ and $S_2=(2)$.
So, in Hypothesis~C,
both $S(\phi(\partial D_i^+))$ and $S(\phi(\partial D_i^-))$
$(\ell_1, n_1 \langle 3 \rangle, 2, n_2 \langle 3 \rangle, \ell_2)$,
where $0 \le \ell_1, \ell_2 \le 2$ and $0 \le n_1, n_2 \le n-1$ are integers
such that a pair $(\ell_j, n_j)$ cannot be $(0, 0)$ for $j=1,2$.
}
\end{remark}

\begin{lemma}
\label{lem:case1-2}
Let $r=2/5=[2,2]$.
Under Hypothesis~C, the following hold for every $i$.
\begin{enumerate}[\indent \rm (1)]
\item $S(z_iy_{i+1})=(2, 2)$ is not possible.

\item $S(z_iy_{i+1})=(2)$ with $|z_i|=|y_{i+1}|=1$ is not possible.
\end{enumerate}
\end{lemma}

\begin{proof}
(1) Suppose on the contrary that $S(z_1y_2)=(2, 2)$.
Since $1 \le |z_1|, |y_2| \le 2$, we have $|z_1|=|y_2|=2$.
Suppose first that $J=M$.
Then it follows from Figure~\ref{fig.lemma-2-1}(a)
that $CS(\phi(\delta^{-1}))=CS(s')$ involves two consecutive $1$'s.
It then follows that $|z_2|=2$,
for otherwise, $|z_2|=1$ and hence
we see $S(y_2'w_2'z_2')=(1,2,2)$,
which in turn implies that
$CS(s')$ contains $(2, 2+c)$ with $c\in \ZZ_+\cup\{0\}$
as a subsequence,
contradicting \cite[Lemma~3.8]{lee_sakuma_2},
because $CS(s')$ also contains $(1,1)$ as a subsequence.
We next observe that $z_2$ and $y_3$ have different signs,
as depicted in Figure~\ref{fig.lemma-2-1}(b).
If otherwise, we have $S(y_1'w_1'z_2'y_3')=(1,2,1+(3-d))=(1,2,4-d)$
where $d=|y_3|\in\{1,2\}$,
and hence $CS(s')$ contains $(2,2+c)$
with $c\in \ZZ_+\cup\{0\}$ as a subsequence,
again contradicting \cite[Lemma~3.8]{lee_sakuma_2}.
Hence $S(z_2y_3)=(2,d)$ with $d=|y_3|\in\{1,2\}$
(see Figure~\ref{fig.lemma-2-1}(c)).
If $d=1$, then we have $S(z_2'y_3'w_3')=(1,3-d,2)=(1,2,2)$,
which again yields a contradiction to \cite[Lemma~3.8]{lee_sakuma_2}.
Hence we see $S(z_2y_3)=(2,2)$.
By repeating this argument,
we see $S(z_iy_{i+1})=(2, 2)$
with $|z_i|=|y_{i+1}|=2$ for every $i$.
But then $CS(\phi(\alpha))=CS(s)$ becomes
$\lp \ell \langle 2 \rangle \rp$ with $\ell \ge 3$,
yielding a contradiction to \cite[Lemma~3.8]{lee_sakuma_2}.
Suppose next that $J \subsetneq M$
(see Figure~\ref{fig.lemma-2-1(2)}).
Note that the assumption $S(z_1y_2)=(2, 2)$ implies that
$S(w_1'z_1'y_2'w_2')=(2,1,1,2)$.
By using this fact, we can see that $|\phi(e_2')|=|\phi(e_3')|=1$,
for otherwise a subsequence of the form $(\ell_1, 1, \ell_2)$
with $\ell_1, \ell_2 \in \ZZ_+$
would occur in $S(\phi(e_2'e_3'))$, which in turn implies that
$CS(\phi(\partial D_1'))=CS(2/5)$ would contain a term $1$,
a contradiction to $CS(2/5)=\lp 3, 2, 3, 2 \rp$.
Assuming that $e_2', e_3', {e_3''}^{-1}, {e_2''}^{-1}$ is
a boundary cycle of $D_1'$,
we have $S(\phi(e_2''e_3''))=(1, 3, 2, 2)$
as depicted in Figure~\ref{fig.lemma-2-1(2)}.
But this is impossible, because
\cite[Corollary~3.25(2)]{lee_sakuma_2} shows that
any subword $w$ of the cyclic word $(\phi(\partial D_1')^{-1})=(u_{2/5}^{\pm 1})$
with $S(w)=(1, 3, 2, 2)=(1,S_1,S_2,2)$
cannot be a product of two pieces.

\begin{figure}[h]
\includegraphics{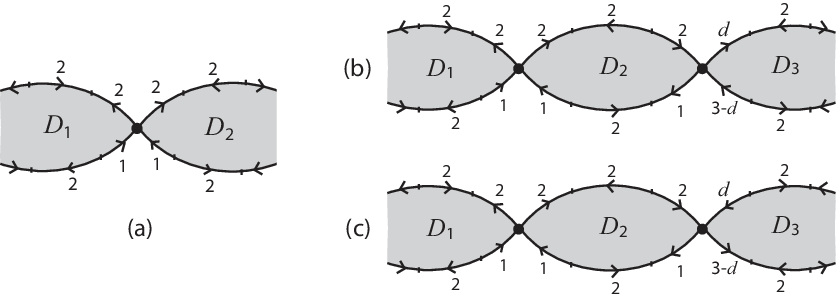}
\caption{
Lemma~\ref{lem:case1-2}(1) where $S(z_1y_2)=(2, 2)$ and $J=M$}
\label{fig.lemma-2-1}
\end{figure}

\begin{figure}[h]
\includegraphics{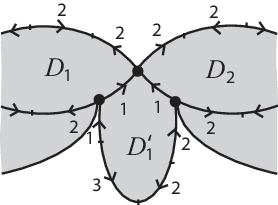}
\caption{
Lemma~\ref{lem:case1-2}(1) where $S(z_1y_2)=(2, 2)$ and $J \subsetneq M$}
\label{fig.lemma-2-1(2)}
\end{figure}

(2) Suppose on the contrary that $S(z_1y_2)=(2)$ with $|z_1|=|y_2|=1$.
Here, if $J=M$ (see Figure~\ref{fig.lemma-2-2}(a)),
then $CS(\phi(\delta^{-1}))=CS(s')$ contains both a term $2$ and
a term $4$, contradicting \cite[Lemma~3.8]{lee_sakuma_2}.
On the other hand, if $J \subsetneq M$,
then, by Lemma~\ref{lem:vertex_position}(2),
none of $S(\phi(e_j'))$ contains $S_2$ in its interior.
This implies that the initial vertex of $e_2'$ lies in
the (central) segment of $\partial D_1^-$ corresponding to $S_2=(2)$
and that the terminal vertex of $e_3'$ lies in
the (central) segment of $\partial D_2^-$ corresponding to $S_2=(2)$.
Then we see that $CS(\phi(\partial D_1'))=CS(2/5)$
contains a term of the form
$4+c$ with $c \in \ZZ_+ \cup \{0\}$,
as illustrated in Figure~\ref{fig.lemma-2-2}(b), a contradiction.
\end{proof}

\begin{figure}[h]
\includegraphics{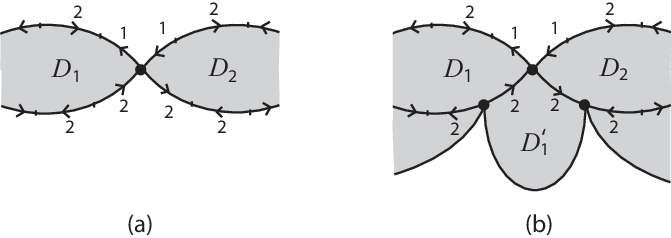}
\caption{
Lemma~\ref{lem:case1-2}(2) where $S(z_1y_2)=(1+1)$}
\label{fig.lemma-2-2}
\end{figure}

\begin{lemma}
\label{lem:case2-2}
Let $r=n/(2n+1)=[2,n]$, where $n \ge 3$ is an integer.
Under Hypothesis~C, the following hold for every $i$.
\begin{enumerate}[\indent \rm (1)]
\item $S(z_iy_{i+1})=(2, 2)$ is not possible.

\item $S(z_iy_{i+1})=(2)$ with $|z_i|=|y_{i+1}|=1$ is not possible.
\end{enumerate}
\end{lemma}

\begin{proof}
The proofs of (1) and (2) are analogous to those of
(1) and (2) in Lemma~\ref{lem:case1-2}.
\end{proof}

\begin{lemma}
\label{lem:case3-2}
Let $r=(n+1)/(3n+2)=[2,1,n]$, where $n \ge 2$ is an integer.
Under Hypothesis~C, the following hold for every $i$.
\begin{enumerate}[\indent \rm (1)]
\item $S(z_iy_{i+1})=(1, 2)$ is not possible,
nor is $S(z_iy_{i+1})=(2, 1)$ possible.

\item $S(z_iy_{i+1})=(n_1 \langle 3 \rangle, 2, n_2 \langle 3 \rangle)$
with $n_1, n_2\in \ZZ_+\cup\{0\}$ is impossible.
\end{enumerate}
\end{lemma}

\begin{proof}
(1) Suppose on the contrary that $S(z_1y_2)=(1, 2)$.
(The other case is proved similarly.)
Then $|z_1|=1$ and $|y_2|=2$.
Here, if $J=M$, then $CS(\phi (\delta^{-1}))=CS(s')$
contains both a term $1$ and a term $3$
by Figure~\ref{fig.III-lemma-2-2}(a),
contradicting \cite[Lemma~3.8]{lee_sakuma_2}.
On the other hand,
if $J \subsetneq M$, then by Lemma~\ref{lem:vertex_position}(2)
the initial vertex of $e_2'$ lies in
the segment of $\partial D_1^-$ corresponding to $S_2=(2)$
and the terminal vertex of $e_3'$ lies in
the segment of $\partial D_2^-$ corresponding to $S_2=(2)$.
This implies that
a subsequence of the form $(\ell_1, 2, 1, \ell_2)$
with $\ell_1, \ell_2 \in \ZZ_+$
occurs in $S(\phi(e_2'e_3'))$ (see Figure~\ref{fig.III-lemma-2-2}(b)),
so in $CS(\phi(\partial D_1'))=CS(r)$,
a contradiction.

\begin{figure}[h]
\includegraphics{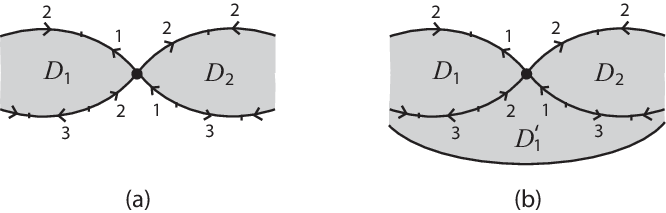}
\caption{
Lemma~\ref{lem:case3-2}(1) where $S(z_1y_2)=(1, 2)$}
\label{fig.III-lemma-2-2}
\end{figure}

(2) Suppose on the contrary that
$S(z_1y_2)=(n_1 \langle 3 \rangle, 2, n_2 \langle 3 \rangle)$.
Then one of the following holds.
\begin{enumerate}[\indent \rm (i)]
\item Either both $S(z_1)=(n_1 \langle 3 \rangle)$ and
$S(y_2)=(2, n_2 \langle 3 \rangle)$,
or both
$S(z_1)=(n_1 \langle 3 \rangle, 2)$ and
$S(y_2)=(n_2 \langle 3 \rangle)$
(see Figure~\ref{fig.III-lemma-2-3}(a)).

\item $S(z_1)=(n_1 \langle 3 \rangle,1)$ and
$S(y_2)=(1, n_2 \langle 3 \rangle)$
(see Figure~\ref{fig.III-lemma-2-3}(b)).
\end{enumerate}
Suppose $J=M$.
Then $CS(\phi (\delta^{-1}))=CS(s')$
contains a term $4$ by Figure~\ref{fig.III-lemma-2-3}.
On the other hand,
by the assumption that Hypothesis~C holds,
$CS(s')$ also contains a term $2$.
This contradicts \cite[Lemma~3.8]{lee_sakuma_2}.
Suppose $J \subsetneq M$.
Then, by using Lemma~\ref{lem:vertex_position}(2)
as in the proof of Lemma~\ref{lem:case3-2}(1),
we see that
a term $4$ occurs in $S(\phi(e_2'e_3'))$.
This implies that $CS(\phi(\partial D_1'))=CS(r)$ contains a term of the form
$4+c$ with $c \in \ZZ_+ \cup \{0\}$, a contradiction.
\end{proof}

\begin{figure}[h]
\includegraphics{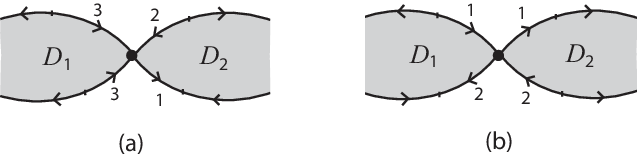}
\caption{
Lemma~\ref{lem:case3-2}(2) where (a) (i) occurs,
and (b) (ii) occurs}
\label{fig.III-lemma-2-3}
\end{figure}

\section{Proof of Main Theorem~\ref{main_theorem} for $K(2/5)$}
\label{proof_for_twist_links_1}

Suppose $r=2/5=[2,2]$. Recall from \cite[Example~3.17(2)]{lee_sakuma_2}
that $CS(2/5)=\lp 3, 2, 3, 2 \rp$, where $S_1=(3)$ and $S_2=(2)$.
For two distinct elements $s, s' \in I_1(2/5) \cup I_2(2/5)$,
suppose on the contrary that the unoriented loops
$\alpha_s$ and $\alpha_{s'}$ are homotopic in $S^3-K(2/5)$,
namely we suppose Hypothesis~A.
We will derive a contradiction in each case to consider.
By Lemma~\ref{lem:boundary_layer}, there are two big cases
to consider.

\medskip
\noindent
{\bf Case 1.} {\it Hypothesis~B holds.}
\medskip

By Corollary~\ref{cor:case1-1},
Case~1 is reduced to the following five cases.

\medskip
\noindent
{\bf Case 1.a.} {\it $CS(s)=\lp 5, 5 \rp$.}
\medskip

Let $v'$ and $v''$ be subwords of the cyclic word
$(\phi(\alpha))=(u_s)$ such that
$(v'v'')=(u_s)$ and $S(v')=S(v'')=(5)$.
Then by using Lemma~\ref{lem:case1-1(a)}
and the facts that $0\le |z_i|, |y_i|\le 2$
and $|w_i|=3$,
we may assume, after a cyclic shift of indices,
that $v'=w_1z_1y_2$ where $|z_1|=0$ and $|y_2|=2$.
This implies $v''=w_2z_2y_1$ where $|z_2|=0$ and $|y_1|=2$.
Thus $J$ is as illustrated in Figure~\ref{fig.case-1a}(a).
If $J=M$,
then $CS(\phi(\delta^{-1}))=CS(s')$ also becomes $\lp 5, 5 \rp$,
which gives $s'=s$, contradicting the hypothesis of the theorem.
Suppose $J \subsetneq M$.
By using Lemma~\ref{lem:vertex_position}(1)
as in the last step of
the proof of Lemma~\ref{lem:case1-1(b)}(1),
we see that
the interior of each of the two
segments of $\partial D_i^-$ with weight $3$
contains a (unique) vertex of $M$.
Moreover by using the fact that $CS(2/5)$ consists of $3$ and $2$,
we see that the position of two vertices is as illustrated in
Figure~\ref{fig.case-1a}(b).
This implies that
$J\cup J_1$ is as illustrated in
Figure~\ref{fig.case-1a}(b),
where $J_1$ is the outer boundary layer of $M-J$.
Thus the inner boundary label of $J_1$ is again $\lp 5, 5 \rp$.
By repeating this argument, we see that
the cyclic $S$-sequence of an inner boundary label of $M$,
namely $CS(s')$,
also becomes $\lp 5, 5 \rp$ regardless of the number of layers of $M$,
and so $s'=s$, contradicting
the hypothesis of the theorem.

\begin{figure}[h]
\includegraphics{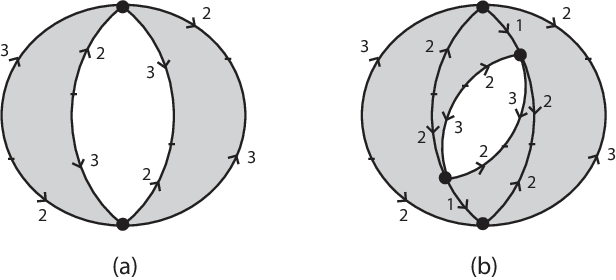}
\caption{
Case 1.a}
\label{fig.case-1a}
\end{figure}

\medskip
\noindent
{\bf Case 1.b.} {\it $CS(s)$ consists of $2$ and $3$.}
\medskip

By \cite[Proposition~3.19(1)]{lee_sakuma_2},
$s \notin I_1(2/5) \cup I_2(2/5)$, contradicting the hypothesis of the theorem.

\medskip
\noindent
{\bf Case 1.c.} {\it $CS(s)$ consists of $3$ and $4$.}
\medskip

By Lemmas~\ref{lem:case1-1(a)},~\ref{lem:case1-1}(1) and (2),
$CS(s)$ must be $\lp 4, 3, 4, 3 \rp$.
Then by Lemma~\ref{lem:case1-1(b)}(2),
there is only one possibility:
$J$ consists of two $2$-cells and
\[
\begin{aligned}
CS(\phi(\alpha))=CS(\phi(\partial D_1^+ \partial D_2^+))
&=\lp S(z_2y_1), S(w_1), S(z_1y_2), S(w_2) \rp \\
&=\lp 4, 3, 4, 3 \rp,
\end{aligned}
\]
where $|y_j|=|z_j|=2$ for $j=1, 2$.
Here, if $J=M$ (see Figure~\ref{fig.case-1c}(a)),
then $CS(\phi(\delta^{-1}))=CS(s')$ becomes $\lp 6 \rp$,
contradicting \cite[Lemma~3.8]{lee_sakuma_2}.
On the other hand,
if $J \subsetneq M$,
then, by an argument similar to that in Case 1.a,
we see that $J\cup J_1$ is as illustrated in
Figure~\ref{fig.case-1a}(b),
where $J_1$ is the outer boundary layer of $M-J$.
Thus if the number of layers of $M$ is two,
then $CS(s')$ also becomes $\lp 4, 3, 4, 3 \rp$, which gives $s'=s$,
contradicting the hypothesis of the theorem.
If the number of layers of $M$ is bigger than $2$,
then, by an argument
using Lemma~\ref{lem:vertex_position} as in the last step of
the proof of Lemma~\ref{lem:case1-1(b)}(1),
a subsequence of the form $(\ell_1,4,\ell_2)$
with $\ell_1, \ell_2 \in \ZZ_+$
occurs in the inner boundary of $J_1$,
so in $CS(2/5)= \lp 3,2,3,2 \rp$, a contradiction.

\begin{figure}[h]
\includegraphics{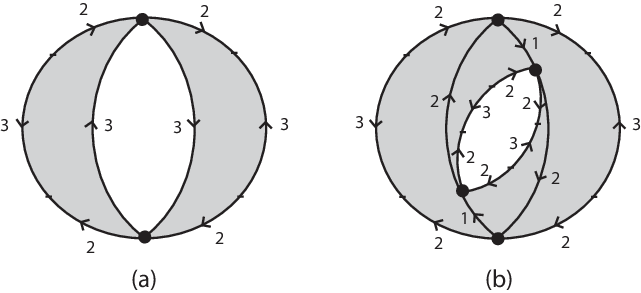}
\caption{
Case 1.c}
\label{fig.case-1c}
\end{figure}

\medskip
\noindent
{\bf Case 1.d.} {\it $CS(s)$ consists of $4$ and $5$.}
\medskip

By Lemmas~\ref{lem:case1-1(a)} and \ref{lem:case1-1(b)}(2),
every subword of the cyclic word
$(u_s)=(\phi(\alpha))$ corresponding to a term $4$ in $CS(s)$
must be $z_jy_{j+1}$ with $|z_j|=|y_{j+1}|=2$ for some $j$.
Without loss of generality, we may assume that $S(z_1y_2)=(4)$,
where $|z_1|=|y_2|=2$.
Since $CS(s)$ consists of $4$ and $5$, we have either $S(w_2z_2y_3)=(4)$
or $S(w_2z_2y_3)=(5)$.
But by Lemma~\ref{lem:case1-1(b)}(2),
$S(w_2z_2y_3)=(5)$,
where $|z_2|=0$ and $|y_3|=2$.
By the repetition of this argument, we have
$S(w_iz_iy_{i+1})=(5)$, where $|z_i|=0$ and $|y_{i+1}|=2$ for every $i$.
This is obviously a contradiction, since we assumed that $|z_1|=2$.

\medskip
\noindent
{\bf Case 1.e.} {\it $CS(s)$ consists of $5$ and $6$.}
\medskip

In this case, $J=M$ by
Lemma~\ref{lem:case1-1}(3b).
By Lemmas~\ref{lem:case1-1(a)} and \ref{lem:case1-1(b)}(2),
every subword of the cyclic word
$(u_s)=(\phi(\alpha))$ corresponding to a term $6$ in $CS(s)$
must be $w_jw_{j+1}$, where $|z_j|=|y_{j+1}|=0$, for some $j$
(cf. proof of Lemma~\ref{lem:case1-1}(3a)).
Without loss of generality, we may assume that $S(w_1w_2)=(6)$,
where $|z_1|=|y_2|=0$.
Since $CS(s)$ consists of $5$ and $6$,
we have the following three possibilities
by Lemmas~\ref{lem:case1-1(a)} and \ref{lem:case1-1(b)}(2):
\begin{enumerate}[\indent (i)]
\item $S(z_2w_3)=(5)$, where $|z_2|=2$ and $|y_3|=0$;

\item $S(w_3y_4)=(5)$, where $|z_2|=|y_3|=|z_3|=0$ and $|y_4|=2$;

\item $S(w_3w_4)=(6)$, where $|z_2|=|y_3|=|z_3|=|y_4|=0$.
\end{enumerate}
If (ii) or (iii) occurs, then $S(w_2z_2y_3w_3)=(3,3)$,
contradicting Lemma~\ref{lem:case1-1(b)}(6).
Hence only (i) can occur. By the repetition of this argument, we have
$S(z_iw_{i+1})=(5)$, where $|z_i|=2$ and $|y_{i+1}|=0$, for every $i$.
This is obviously a contradiction, since we assumed $|z_1|=0$.

\medskip
\noindent
{\bf Case 2.} {\it Hypothesis~C holds.}
\medskip

By Remark~\ref{rem:(2)holds}(1),
the cyclic $S$-sequence $CS(\phi(\alpha))=CS(s)$ contains a term $2$.
Hence, by \cite[Lemma~3.8]{lee_sakuma_2}, Case~2 is reduced to the following three subcases:
if $CS(s)$ has the form $\lp m, m \rp$, then $m$ is $2$,
while if $CS(s)$ has the form consisting of $m$ and $m+1$,
then $m$ is either $1$ or $2$.

\medskip
\noindent
{\bf Case 2.a.} {\it $CS(s)=\lp 2, 2 \rp$.}
\medskip

In this case, there is only one possibility:
$J$ consists of one $2$-cell, namely
$CS(\phi(\alpha))=CS(\phi(\partial D_1^+))=
\lp S(z_0y_1), S(w_1) \rp=\lp 2, 2 \rp$
with $|y_1|=|z_0|=1$. But this contradicts Lemma~\ref{lem:case1-2}(2).

\medskip
\noindent
{\bf Case 2.b.} {\it $CS(s)$ consists of $1$ and $2$.}
\medskip

Recall from Remark~\ref{rem:(2)holds}(1) that
$S(\phi(\partial D_i^+))=(S(y_i),S(w_i),S(z_i))=
(\ell_{i,1}, 2, \ell_{i,2})$, where $1 \le \ell_{i,j} \le 2$
are integers.
By using Lemma~\ref{lem:case1-2}(2) and
the assumption that $CS(s)$ consists of $1$ and $2$,
we see $S(z_iy_{i+1})=(\ell_{i,2},\ell_{i+1,1})$
and therefore
$CS(s)= \lp \ell_{1,1}, 2, \ell_{1,2},
\dots, \ell_{t,1}, 2, \ell_{t,2} \rp$.
By Lemma~\ref{lem:case1-2}(1),
$(\ell_{i, 2}, \ell_{i+1, 1})$ is one of
$(1, 1)$, $(1, 2)$ and $(2, 1)$ for every $i$.
Thus if the number $t$ of the $2$-cells of $J$ is one,
then $CS(\phi(\alpha))=CS(s)$
is either $\lp 1, 2, 2 \rp$ or $\lp 1, 2, 1 \rp$,
both yielding a contradiction to \cite[Proposition~3.19(1)]{lee_sakuma_2}.
Hence $t\ge 2$.

First, assume that $S(z_1y_2)=(\ell_{1, 2}, \ell_{2, 1})=(1, 1)$.
Then $(\ell_{i, 2}, \ell_{i+1, 1})$ is $(1, 1)$ for every $i$,
for otherwise $CS(s)$
would contain consecutive $1$'s and consecutive $2$'s,
contradicting \cite[Lemma~3.8]{lee_sakuma_2}.
Then we have $CT(s)= \lp t\langle 2\rangle \rp$ and therefore
we have $t=2$ by \cite[Lemma~3.8]{lee_sakuma_2} and
\cite[Corollary~3.14]{lee_sakuma_2}.
Thus $J$ is as illustrated in Figure~\ref{fig.case-2b-iii}(a)
and we have
$CS(\phi(\alpha))=CS(\phi(\partial D_1^+ \partial D_2^+))
=\lp 1, 2, 1, 1, 2, 1 \rp = \lp 2, 1, 1, 2, 1, 1 \rp$.
If $J=M$, then $CS(\phi(\delta^{-1}))=CS(s')=\lp 2, 2, 2, 2, 2, 2 \rp$,
contradicting \cite[Lemma~3.8]{lee_sakuma_2}.
On the other hand, if $J \subsetneq M$,
then, by using Lemma~\ref{lem:vertex_position}(2)
as in the last step of
the proof of Lemma~\ref{lem:case1-1(b)}(1),
we see that the unique vertex of $M$
in the interior of $\partial D_i^-$ must lie in the
central segment among the three segments of $\partial D_i^-$
with weight $2$ for each $i=1,2$.
By using this fact and the identity $CS(2/5)=\lp 3,2,3,2 \rp$,
we see that $J\cup J_1$, where $J_1$ is the outer layer of
$M-J$, is as illustrated in Figure~\ref{fig.case-2b-iii}(b).
If the number of layers of $M$ is two,
then $CS(s')$ also becomes $\lp 1, 2, 1, 1, 2, 1 \rp$,
which gives $s'=s$, contradicting the hypothesis of the theorem.
If the number of layers of $M$ is bigger than $2$,
then by an argument using Lemma~\ref{lem:vertex_position}(2)
as in the last step of the proof of Lemma~\ref{lem:case1-2}(2),
$CS(2/5)$ would contain a term $1$, a contradiction.

\begin{figure}[h]
\includegraphics{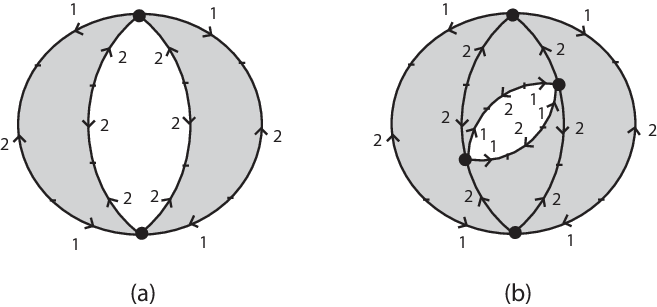}
\caption{
Case 2.b where $S(z_1y_2)=(1, 1)$}
\label{fig.case-2b-iii}
\end{figure}

Next, assume that $S(z_1y_2)=(\ell_{1, 2}, \ell_{2, 1})=(1, 2)$.
(The case for $S(z_1y_2)=(2, 1)$ is similar.)
Then $(\ell_{i, 2}, \ell_{i+1, 1})$ is $(1, 2)$ for every $i$,
for otherwise there would exist some $j$ and $j'$ such that
$(\ell_{j-1, 2}, \ell_{j, 1}, 2, \ell_{j, 2}, \ell_{j+1, 1})=(1, 2, 2, 2, 1)$
and such that
$(\ell_{j'-1, 2}, \ell_{j', 1}, 2, \ell_{j', 2}, \ell_{j'+1, 1})=(2, 1, 2, 1, 2)$,
so $CT(s)$ would contain both a term $1$ and a term $3$,
which together with \cite[Corollary~3.14]{lee_sakuma_2} would yield
a contradiction to \cite[Lemma~3.8]{lee_sakuma_2}.
Thus $CS(s)=\lp 2,2,1,\dots,2,2,1 \rp$ and therefore
$CT(s)= \lp t\langle 2 \rangle \rp$.
Hence, by \cite[Lemma~3.8]{lee_sakuma_2} and \cite[Corollary~3.14]{lee_sakuma_2},
we have $t=2$.
Thus $J$ is as illustrated in Figure~\ref{fig.case-2b-i}(a), and
$CS(\phi(\alpha))=CS(\phi(\partial D_1^+ \partial D_2^+))
=\lp 2, 2, 1, 2, 2, 1 \rp$.
Here, if $J=M$,
then $CS(\delta^{-1})=CS(s')=\lp 1, 2, 2, 1, 2, 2 \rp$,
which gives $s'=s$, contradicting the hypothesis of the theorem.
On the other hand,
if $J \subsetneq M$,
then, by using Lemma~\ref{lem:vertex_position}(2)
(cf. the last step of the proof of Lemma~\ref{lem:case1-2}(2))
and the identity $CS(2/5)=\lp 3,2,3,2 \rp$,
we see that $J\cup J_1$, where $J_1$ is the outer layer of
$M-J$, is as illustrated in Figure~\ref{fig.case-2b-i}(b).
By repeating this argument,
we see that
the cyclic $S$-sequence of an inner boundary label of $M$,
namely $CS(s')$, is $\lp 1, 2, 2, 1, 2, 2 \rp$
regardless of the number of layers of $M$, which also gives $s'=s$,
contradicting the hypothesis of the theorem.

\begin{figure}[h]
\includegraphics{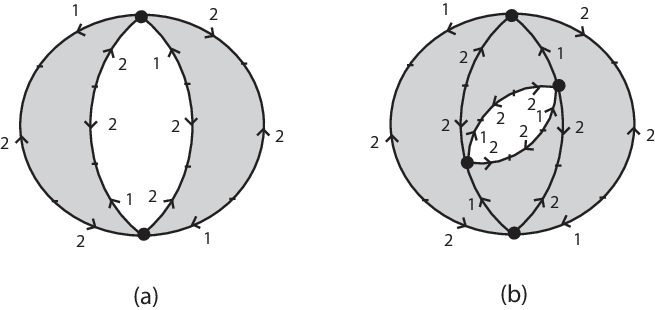}
\caption{
Case 2.b where $S(z_1y_2)=(1, 2)$}
\label{fig.case-2b-i}
\end{figure}

\medskip
\noindent
{\bf Case 2.c.} {\it $CS(s)$ consists of $2$ and $3$.}
\medskip

In this case, by \cite[Proposition~3.19(1)]{lee_sakuma_2},
$s \notin I_1(2/5) \cup I_2(2/5)$, contradicting the hypothesis of the theorem.
\qed

\section{Proof of Main Theorem~\ref{main_theorem}
for $K(n/(2n+1))$ with $n \ge 3$}
\label{proof_for_twist_links_2}

Suppose $r=n/(2n+1)=[2,n]$, where $n \ge 3$ is an integer.
Recall from \cite[Example~3.17(2)]{lee_sakuma_2}that
$CS(r)=\lp 3, (n-1) \langle 2 \rangle, 3, (n-1) \langle 2 \rangle \rp$,
where $S_1=(3)$ and $S_2=((n-1) \langle 2 \rangle)$.
For two distinct elements $s, s' \in I_1(r) \cup I_2(r)$,
suppose on the contrary that the unoriented loops
$\alpha_s$ and $\alpha_{s'}$ are homotopic in $S^3-K(r)$,
namely we suppose Hypothesis~A.
We will derive a contradiction in each case to consider.
By Lemma~\ref{lem:boundary_layer}, there are two big cases
to consider.

\medskip
\noindent
{\bf Case 1.} {\it Hypothesis~B holds.}
\medskip

By Hypothesis~B together with Remark~\ref{rem:(1)holds}(2),
$\phi(\alpha)$ involves a subword $w_i$ whose $S$-sequence is $(3)$,
so the cyclic $S$-sequence $CS(\phi(\alpha))=CS(u_s)=CS(s)$ must contain
a term of the form $3+c$, where $c \in \ZZ_+ \cup \{0\}$.
This together with \cite[Lemma~3.8]{lee_sakuma_2} and Lemma~\ref{lem:case2-1} implies that
$CS(s)$ cannot have the form $\lp m, m \rp$
and that Case~1 is reduced to the following two subcases:
either $CS(s)$ consists of $2$ and $3$
or $CS(s)$ consists of $3$ and $4$.

\medskip
\noindent
{\bf Case 1.a.} {\it $CS(s)$ consists of $2$ and $3$.}
\medskip

Without loss of generality,
we may assume that $2$ occurs in $S(z_1y_2)$.
There are three possibilities:
\begin{enumerate}[\indent \rm (i)]
\item
$S(z_1y_2)$ consists of only $2$, where
$S(z_1)=(n_1\langle 2 \rangle)$,
$S(y_2)=(n_2\langle 2 \rangle)$,
and $S(z_1y_2)=((n_1+n_2)\langle 2 \rangle)$
with $n_1,n_2\in\ZZ_+\cup\{0\}$;

\item
$S(z_1y_2)$ consists of only $2$, where
$S(z_1)=(n_1\langle 2 \rangle,1)$,
$S(y_2)=(1,n_2\langle 2 \rangle)$,
and $S(z_1y_2)=((n_1+n_2+1)\langle 2 \rangle)$
with $n_1,n_2\in\ZZ_+\cup\{0\}$;

\item
$S(z_1y_2)$ consists of $2$ and $3$.
\end{enumerate}

First assume that (i) occurs.
Then $S(z_1'y_2')=((n_1'+n_2')\langle 2 \rangle)$
where $n_1'=(n-1)-n_1$ and $n_2'=(n-1)-n_2$.
So $n_1'+n_2'=2(n-1)-(n_1+n_2)$ and hence
either $S(z_1y_2)$ or $S(z_1'y_2')$ contains
$n-1$ consecutive $2$'s.
If $J=M$, then this implies that either $s \notin I_1(r) \cup I_2(r)$
or $s' \notin I_1(r) \cup I_2(r)$ by \cite[Proposition~3.19(1)]{lee_sakuma_2},
contradicting the hypothesis of the theorem.
On the other hand, if $J \subsetneq M$, then
the above observation implies that
either $S(z_1y_2)$
contains $n-1$ consecutive $2$'s and so $s \notin I_1(r) \cup I_2(r)$,
or otherwise $S(z_1'y_2')$ contains $n$ consecutive $2$'s.
The former case is impossible by the assumption.
In the latter case,
we see, by an argument using Lemma~\ref{lem:vertex_position}(1)
as in the last step of the proof of Lemma~\ref{lem:case1-1(b)}(1),
that a subsequence of the form $(\ell_1, n \langle 2 \rangle, \ell_2)$
with $\ell_1, \ell_2 \in \ZZ_+$
occurs in $S(\phi(e_2'e_3'))$, so in $CS(\phi(\partial D_1'))=CS(r)$,
a contradiction.

Next assume that (ii) occurs.
Then $S(z_1'y_2')=((n_1'+n_2'+1)\langle 2 \rangle)$,
where $n_1'=(n-2)-n_1$ and $n_2'=(n-2)-n_2$.
By using the identity
$n_1'+n_2'+1=2(n-1)-(n_1+n_2+1)$,
this case is treated as in the case when (i) occurs.

Finally assume that (iii) occurs.
If $J=M$ (see Figure~\ref{fig.II-case-1a1}(a)),
then $CS(\phi(\delta^{-1}))=CS(s')$ includes
both a term $1$ and a term of the form $3+c$
with $c \in \ZZ_+ \cup \{0\}$, contradicting \cite[Lemma~3.8]{lee_sakuma_2}.
On the other hand, if $J \subsetneq M$ (see Figure~\ref{fig.II-case-1a1}(b)),
then, by an argument using Lemma~\ref{lem:vertex_position}(1)
as in the last step of
the proof of Lemma~\ref{lem:case1-1(b)}(1),
a subsequence of the form $(\ell_1, 1, \ell_2)$
with $\ell_1, \ell_2 \in \ZZ_+$ occurs in $S(\phi(e_2'e_3'))$,
so in $CS(\phi(\partial D_1'))=CS(r)$, a contradiction.

\begin{figure}[h]
\includegraphics{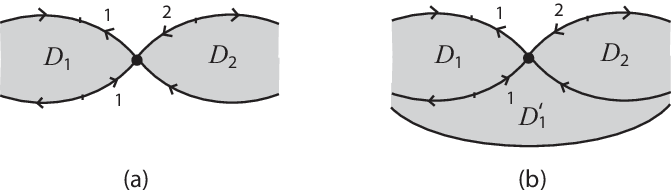}
\caption{
Case 1.a where (iii) occurs}
\label{fig.II-case-1a1}
\end{figure}

\medskip
\noindent
{\bf Case 1.b.} {\it $CS(s)$ consists of $3$ and $4$.}
\medskip

By Lemma~\ref{lem:case2-1}(1) and (2),
$CS(s)$ must be $\lp 4, 3, 4, 3 \rp$.
Then by Lemmas~\ref{lem:case1-1(a)} and \ref{lem:case2-1(b)}(2),
there is only one possibility:
$J$ consists of two $2$-cells, namely
$CS(\phi(\alpha))=CS(\phi(\partial D_1^+ \partial D_2^+))
=\lp S(z_2y_1), S(w_1), S(z_1y_2), S(w_2) \rp
=\lp 4, 3, 4, 3 \rp$ with $|y_j|=|z_j|=2$ for each $j=1, 2$.

First suppose $n=3$.
If $J=M$ (see Figure~\ref{fig.II-case-1b}(a)),
then $CS(\phi(\delta^{-1}))=CS(s')$ also becomes $\lp 4, 3, 4, 3 \rp$
implying that $s'=s$, a contradiction to the hypothesis of the theorem.
On the other hand, if $J \subsetneq M$,
then, by an argument using Lemma~\ref{lem:vertex_position}(1)
as in the last step of
the proof of Lemma~\ref{lem:case1-1(b)}(1),
a subsequence of the form $(\ell_1, 4, \ell_2)$
with $\ell_1, \ell_2 \in \ZZ_+$
occurs in $S(\phi(e_2'e_3'))$, so in $CS(\phi(\partial D_1'))=CS(r)$,
a contradiction.

\begin{figure}[h]
\includegraphics{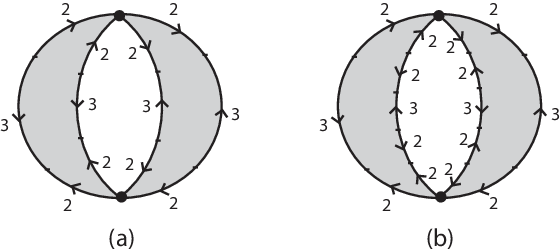}
\caption{
Case 1.b where (a) $r=[2, 3]$ and (b) $r=[2, n]$ with $n=4$}
\label{fig.II-case-1b}
\end{figure}

Next suppose $n\ge 4$.
If $J=M$ (see Figure~\ref{fig.II-case-1b}(b)),
then $CS(\phi(\delta^{-1}))=CS(s')$ includes both a term $2$
and a term $4$,
contradicting \cite[Lemma~3.8]{lee_sakuma_2}.
On the other hand, if $J \subsetneq M$, then we obtain a contradiction
by the same reason as for $n=3$.

\medskip
\noindent
{\bf Case 2.} {\it Hypothesis~C holds.}
\medskip

By Remark~\ref{rem:(2)holds}(2),
the cyclic $S$-sequence $CS(\phi(\alpha))=CS(s)$
properly contains a subsequence $(2,2)$.
Hence, by \cite[Lemma~3.8]{lee_sakuma_2},
Case~2 is reduced to the following two subcases:
either $CS(s)$ consists of $1$ and $2$
or $CS(s)$ consists of $2$ and $3$.

\medskip
\noindent
{\bf Case 2.a.} {\it $CS(s)$ consists of $1$ and $2$.}
\medskip

Repeat the argument of Case~2.b of Section~\ref{proof_for_twist_links_1}
replacing the reference to Lemma~\ref{lem:case1-2} with the reference
to Lemma~\ref{lem:case2-2} to obtain that
$J$ consists of at least two $2$-cells and that
$S(z_iy_{i+1})$ is one of $(1, 1)$, $(1, 2)$ and $(2, 1)$ for every $i$,
so that $CS(s)= \lp \ell_{1,1}, (n-1) \langle 2 \rangle, \ell_{1,2}, \dots,
\ell_{t,1}, (n-1) \langle 2 \rangle, \ell_{t,2} \rp$,
where $(\ell_{i, 2}, \ell_{i+1, 1})$ is one of
$(1, 1)$, $(1, 2)$ and $(2, 1)$ for every $i$.

First, assume that $S(z_1y_2)=(\ell_{1, 2}, \ell_{2, 1})=(1, 1)$.
Then $CS(s)$ contains consecutive $1$'s
and consecutive $2$'s, contradicting \cite[Lemma~3.8]{lee_sakuma_2}.

Next, assume that $S(z_1y_2)=(\ell_{1, 2}, \ell_{2, 1})=(1, 2)$.
(The case for $S(z_1y_2)=(2, 1)$ is similar.)
The same argument of Case~2.b of Section~\ref{proof_for_twist_links_1}
implies that both $CS(\phi(\alpha))=CS(s)$ and $CS(\phi(\delta^{-1}))=CS(s')$ become
$\lp n \langle 2 \rangle, 1, n \langle 2 \rangle, 1 \rp$
(see Figure~\ref{fig.II-case-2a-i}(a) and (b)),
so that $s'=s$, contradicting the hypothesis of the theorem.

\begin{figure}[h]
\includegraphics{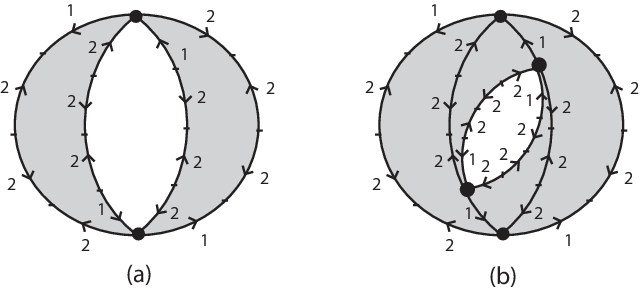}
\caption{
Case 2.a where $S(z_1y_2)=(1, 2)$ and $r=[2, n]$ with $n=3$}
\label{fig.II-case-2a-i}
\end{figure}

\medskip
\noindent
{\bf Case 2.b.} {\it $CS(s)$ consists of $2$ and $3$.}
\medskip

In this case,
$CS(s)$ contains $S_1=(3)$ as a subsequence.
Moreover, $CS(s)$ also contains $S_2=((n-1)\langle 2\rangle)$
by Hypothesis C.
Hence, $s \notin I_1(r) \cup I_2(r)$ by \cite[Proposition~3.19(1)]{lee_sakuma_2},
contradicting the hypothesis of the theorem.
\qed

\section{Proof of Main Theorem~\ref{main_theorem_2}}
\label{proof_for_[2,1,n]}

Let $r=(n+1)/(3n+2)=[2,1,n]$, where $n \ge 2$ is an integer.
Recall from \cite[Example~3.17(1)]{lee_sakuma_2} that
$CS(r)=\lp n \langle 3 \rangle, 2, n \langle 3 \rangle, 2 \rp$,
where $S_1=(n \langle 3 \rangle)$ and $S_2=(2)$.
For two distinct elements $s, s' \in I_1(r) \cup I_2(r)$,
suppose that the unoriented loops
$\alpha_s$ and $\alpha_{s'}$ are homotopic in $S^3-K(r)$,
namely we suppose Hypothesis~A.
We will prove the assertion by showing that
$\alpha_s$ and $\alpha_{s'}$ are homotopic in $S^3-K(r)$
for both $r=3/8$ and the set $\{s, s'\}$ equals
either $\{1/6, 3/10\}$ or $\{3/4, 5/12\}$
and that we obtain a contradiction in the other cases.
By Lemma~\ref{lem:boundary_layer},
there are two big cases to consider.

\medskip
\noindent
{\bf Case 1.} {\it Hypothesis~B holds.}
\medskip

By Hypothesis~B along with Remark~\ref{rem:(1)holds}(3),
$\phi(\alpha)$ involves a subword $w_i$
whose $S$-sequence is $(n \langle 3 \rangle)$,
so the cyclic $S$-sequence $CS(\phi(\alpha))=CS(s)$ must contain
a term of the form $3+c$, where $c \in \ZZ_+ \cup \{0\}$.
This together with \cite[Lemma~3.8]{lee_sakuma_2} and Lemma~\ref{lem:case3-1} implies that
Case~1 is reduced to the following three subcases:
if $CS(s)$ has the form $\lp m, m \rp$, then $m$ is either $3$ or $6$,
while if $CS(s)$ consists of $m$ and $m+1$,
then $m$ is either $2$ or $3$.

\medskip
\noindent
{\bf Case 1.a.} {\it $CS(s)=\lp 3, 3 \rp$.}
\medskip

There is only one possibility:
$n=2$ and $J$ consists of one $2$-cell, namely
$CS(\phi(\alpha))=CS(\phi(\partial D_1^+))=CS(w_1)=\lp 3, 3 \rp$,
where $|y_1|=|z_1|=0$.
Here, if $J=M$ (see Figure~\ref{fig.III-case-1a}(a)),
then $CS(\phi(\delta^{-1}))=CS(s')$ becomes
$\lp 2, 3, 3, 2 \rp$, contradicting \cite[Lemma~3.8]{lee_sakuma_2}.
On the other hand, if $J \subsetneq M$
(see Figure~\ref{fig.III-case-1a}(b)),
then, by an argument using Lemma~\ref{lem:vertex_position}(1)
as in the last step of
the proof of Lemma~\ref{lem:case1-1(b)}(1),
a subsequence of the form $(\ell_1, 2, 2, \ell_2)$
with $\ell_1, \ell_2 \in \ZZ_+$
occurs in $S(\phi(e_2'e_1'))$, so in $CS(\phi(\partial D_1'))=CS(r)$,
a contradiction.

\begin{figure}[h]
\includegraphics{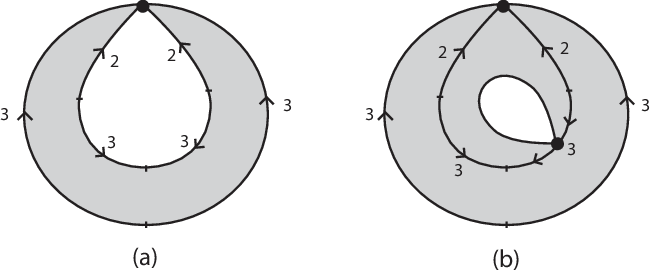}
\caption{
Case 1.a where $r=[2, 1, 2]$}
\label{fig.III-case-1a}
\end{figure}

\medskip
\noindent
{\bf Case 1.b.} {\it $CS(s)=\lp 6, 6 \rp$.}
\medskip

There is only one possibility:
$n=2$ and $J$ consists of two $2$-cells, namely
$CS(\phi(\alpha))=CS(\phi(\partial D_1^+ \partial D_2^+))
=CS(w_1w_2)=\lp 6, 6 \rp$,
where $|y_j|=|z_j|=0$ for $j=1, 2$.
So $r=[2, 1, 2]=3/8$ and $s=1/6$.
Here, if $J=M$ (see Figure~\ref{fig.III-case-1b}(a)),
then $CS(\phi(\delta^{-1}))=CS(s')$ becomes
$\lp 4, 3, 3, 4, 3, 3 \rp$, and so $s'=3/10$
by \cite[Remark~3.15]{lee_sakuma_2}.
This shows that the loops $\alpha_{1/6}$ and $\alpha_{3/10}$ are homotopic
in $S^3-K(3/8)$.
On the other hand, if $J \subsetneq M$ (see Figure~\ref{fig.III-case-1b}(b)),
then, by an argument using Lemma~\ref{lem:vertex_position}(1)
as in the last step of
the proof of Lemma~\ref{lem:case1-1(b)}(1),
a subsequence of the form $(\ell_1, 4, \ell_2)$
with $\ell_1, \ell_2 \in \ZZ_+$
occurs in $S(\phi(e_2'e_3'))$, so in $CS(\phi(\partial D_1'))=CS(r)$,
a contradiction.

\begin{figure}[h]
\includegraphics{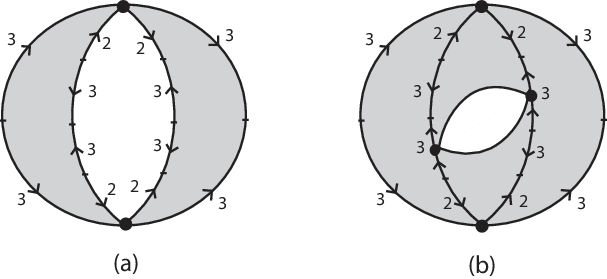}
\caption{
Case 1.b where $r=[2, 1, 2]$}
\label{fig.III-case-1b}
\end{figure}

\medskip
\noindent
{\bf Case 1.c.} {\it $CS(s)$ consists of $2$ and $3$.}
\medskip

In this case,
$CS(s)$ contains $S_2=(2)$ as a subsequence.
Moreover, $CS(s)$ also contains $S_1=(n\langle 3\rangle)$
by Hypothesis~B.
Hence, $s \notin I_1(r) \cup I_2(r)$ by \cite[Proposition~3.19(1)]{lee_sakuma_2},
contradicting the hypothesis of the theorem.
\qed

\medskip
\noindent
{\bf Case 1.d.} {\it $CS(s)$ consists of $3$ and $4$.}
\medskip

We first observe that either
$|z_i| \neq 0$ or $|y_{i+1}| \neq 0$ for every $i$.
Suppose on the contrary that $|z_i|=|y_{i+1}|=0$ for some $i$.
Then, since $CS(s)$ consists of $3$ and $4$ by assumption,
we have $S(w_{i,e}z_iy_{i+1}w_{i+1,b})=S(w_{i,e}w_{i+1,b})=(3,3)$,
contradicting Lemma~\ref{lem:case3-1(b)}(4).

Next, we observe $|z_i| \neq 0$ and $|y_{i+1}| \neq 0$ for every $i$.
Suppose on the contrary that $|z_i|=0$ for some $i$.
(The case $|y_{i+1}|=0$ is treated similarly.)
Then, by the preceding observation
and Remark~\ref{rem:(1)holds}(3),
we see $|y_{i+1}|=1$ or $2$.
Since $S(w_{i,e}z_iy_{i+1})$ is not equal to $(4)$ nor $(5)$
by Lemma~\ref{lem:case3-1(b)}(2) and (3),
this implies that $S(w_{i,e}z_iy_{i+1})$ is equal to
$(3,1)$ or $(3,2)$
and hence $S(w_{i,e}z_iy_{i+1}w_{i+1,b})$
is equal to $(3,1,3)$ or $(3,2,3)$.
This contradicts the assumption that
$CS(s)$ consists of $3$ and $4$.

Thus we have shown that
$|z_i| \neq 0$ and $|y_{i+1}| \neq 0$ for every $i$.
So the assumption that $CS(s)$ consists of $3$ and $4$
implies that $S(z_iy_{i+1})=(3)$ or $(4)$.
However, the former is impossible by Lemma~\ref{lem:case3-1(b)}(1).
Therefore $S(z_iy_{i+1})=(4)$ with $|z_i|=|y_{i+1}|=2$  for every $i$.
Thus $CS(s)=\lp 4,n\langle 3\rangle,\dots,4,n\langle 3\rangle \rp$,
where the subsequence $(4,n\langle 3\rangle)$ appears $t$-times
in $CS(s)$ with $t$ the number of $2$-cells of $J$.
Then $CT(s)=\lp t\langle n\rangle \rp$, and hence $t=2$ by
\cite[Lemma~3.8]{lee_sakuma_2} and \cite[Corollary~3.14]{lee_sakuma_2}.
So, $CS(s)=\lp 4,n\langle 3\rangle,4,n\langle 3\rangle \rp$.

Suppose first that $n=2$.
Then $r=[2,1,2]=3/8$ and $J$ is obtained from the map in
Figure~\ref{fig.III-case-1b}(a) by reversing
the outer and inner boundaries.
Thus $CS(s)=\lp 4, 3, 3, 4, 3, 3 \rp$ and so $s=3/10$,
and the inner boundary label of $J$ is $\lp 6,6 \rp$.
If $J=M$, then $CS(s')=CS(\phi(\delta^{-1})=\lp 6, 6 \rp$, and hence $s'=1/6$.
This again shows that the loops $\alpha_{3/10}$ and $\alpha_{1/6}$ are homotopic
in $S^3-K(3/8)$.
If $J \subsetneq M$,
then, by an argument using Lemma~\ref{lem:vertex_position}(1)
as in the last step of
the proof of Lemma~\ref{lem:case1-1(b)}(1),
$CS(r)$ must contain a term
$4+c$ with $c\in \ZZ_+\cup\{0\}$, a contradiction.

Suppose next that $n \ge 3$.
If $J=M$ (see Figure~\ref{fig.III-case-1e}(a)),
then $CS(\phi(\delta^{-1}))=CS(s')$ contains
both a term $3$ and a term $6$,
contradicting \cite[Lemma~3.8]{lee_sakuma_2}.
On the other hand, if $J \subsetneq M$ (see Figure~\ref{fig.III-case-1e}(b)),
then, by an argument
using Lemma~\ref{lem:vertex_position}(1)
as in the last step of
the proof of Lemma~\ref{lem:case1-1(b)}(1),
a subsequence of the form $(\ell_1, 4+c, \ell_2)$,
where $\ell_1, \ell_2 \in \ZZ_+$ and $c \in \ZZ_+ \cup \{0\}$,
occurs in $S(\phi(e_2'e_3'))$, so in $CS(\phi(\partial D_1'))=CS(r)$,
a contradiction.

\begin{figure}[h]
\includegraphics{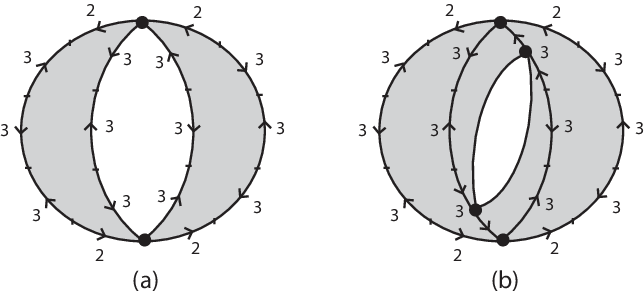}
\caption{
Case 1.d where $r=[2, 1, n]$ with $n=3$}
\label{fig.III-case-1e}
\end{figure}

\medskip
\noindent
{\bf Case 2.} {\it Hypothesis~C holds.}
\medskip

By Remark~\ref{rem:(2)holds}(3),
the cyclic $S$-sequence $CS(\phi(\alpha))=CS(s)$ includes a term $2$.
Hence, by \cite[Lemma~3.8]{lee_sakuma_2}, Case~2 is reduced to the following three subcases:
if $CS(s)$ has the form $\lp m, m \rp$, then $m$ is $2$,
while if $CS(s)$ consists of $m$ and $m+1$,
then $m$ is either $1$ or $2$.

\medskip
\noindent
{\bf Case 2.a.} {\it $CS(s)=\lp 2, 2 \rp$.}
\medskip

There is only one possibility:
$J$ consists of one $2$-cell, namely
$CS(\phi(\alpha))=CS(\phi(\partial D_1^+))=\lp S(z_1y_1), S(w_1) \rp= \lp 2, 2 \rp$.
Then $S(z_1y_1)=(2)$, contradicting
Lemma~\ref{lem:case3-2}(2).

\medskip
\noindent
{\bf Case 2.b.} {\it $CS(s)$ consists of $1$ and $2$.}
\medskip

By Remark~\ref{rem:(2)holds}(3) and the assumption that
$CS(s)$ consists of $1$ and $2$,
we see
$S(\phi(\partial D_i^+))=(S(y_i),S(w_i),S(z_i))=
(\ell_{i,1}, 2, \ell_{i,2})$, where $1 \le \ell_{i,j} \le 2$
are integers.
By using Lemma~\ref{lem:case3-2}(2) and
the assumption that $CS(s)$ consists of $1$ and $2$,
we see $S(z_iy_{i+1})=(\ell_{i,2},\ell_{i+1,1})$
and therefore
$CS(s)= \lp \ell_{1,1}, 2, \ell_{1,2},\dots, \ell_{t,1}, 2, \ell_{t,2} \rp$.
By Lemma~\ref{lem:case3-2}(1),
$(\ell_{i, 2}, \ell_{i+1, 1})$ is either
$(1, 1)$ or $(2, 2)$ for every $i$.
Thus if the number $t$ of the $2$-cells of $J$ is one,
then $CS(\phi(\alpha))=CS(s)$
is either $\lp 1, 2, 1 \rp$ or $\lp 2, 2, 2 \rp$,
both yielding a contradiction to \cite[Proposition~3.19(1)]{lee_sakuma_2}.
Hence $t\ge 2$.
By \cite[Lemma~3.8]{lee_sakuma_2},
we see either
$S(z_iy_{i+1})=(1,1)$ for all $i$
or $S(z_iy_{i+1})=(2,2)$ for all $i$.
However, if the latter holds, then
$CS(s)=\lp 3t\langle 2\rangle \rp$,
a contradiction to \cite[Lemma~3.8]{lee_sakuma_2}.
So, $S(z_iy_{i+1})=(1,1)$ for all $i$,
and therefore $CS(s)=\lp 1,2,1,\dots,1,2,1 \rp$
and $CT(s)=\lp t\langle 2\rangle \rp$.
Hence we have $t=2$
by \cite[Lemma~3.8]{lee_sakuma_2} and \cite[Corollary~3.14]{lee_sakuma_2}.
Thus
$CS(\phi(\alpha))=CS(\phi(\partial D_1^+ \partial D_2^+))
=\lp 1, 2, 1, 1, 2, 1 \rp$.
So $s=3/4$.

First suppose $n=2$, namely $r=3/8$.
If $J=M$ (see Figure~\ref{fig.III-case-2b}(a)),
then $CS(\delta^{-1})=CS(s')$ becomes
$\lp 3, 2, 3, 2, 2, 3, 2, 3, 2, 2 \rp$,
and so $s'=5/12$
by \cite[Remark~3.15]{lee_sakuma_2}.
This shows that the loops $\alpha_{3/4}$ and $\alpha_{5/12}$
are homotopic in $S^3-K(3/8)$.
If $J \subsetneq M$, then
by Lemma~\ref{lem:vertex_position}(2),
the initial vertex of $e_2'$
(resp., the terminal vertex of $e_3'$)
must lie in the central segment of
$\partial D_1^-$ (resp., $\partial D_2^-$) with weight $2$
(see Figure~\ref{fig.III-case-2b}(b)).
Thus a subsequence of the form
$(\ell_1, 2, 2, \ell_2)$
with $\ell_1, \ell_2 \in \ZZ_+$ occurs in $S(\phi(e_2'e_3'))$,
so in $CS(\phi(\partial D_1'))=CS(r)$, a contradiction.

\begin{figure}[h]
\includegraphics{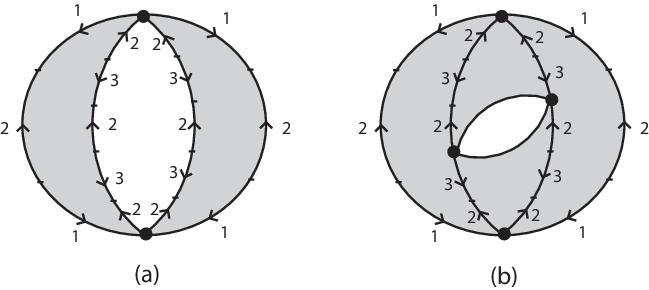}
\caption{
Case 2.b where $r=[2, 1, 2]$}
\label{fig.III-case-2b}
\end{figure}

Now let $n \ge 3$.
If $J=M$, then $CS(s')$ contains consecutive $2$'s and consecutive $3$'s,
contradicting \cite[Lemma~3.8]{lee_sakuma_2}.
On the other hand, if $J \subsetneq M$,
then we see, by using Lemma~\ref{lem:vertex_position}(2)
as in the case $n=2$ (cf. Figure~\ref{fig.III-case-2b}(b)), that
a subsequence of the form $(\ell_1, 2, 2, \ell_2)$
with $\ell_1, \ell_2 \in \ZZ_+$ occurs in $S(\phi(e_2'e_3'))$,
so in $CS(\phi(\partial D_1'))=CS(r)$, a contradiction.

\medskip
\noindent
{\bf Case 2.c.} {\it $CS(s)$ consists of $2$ and $3$.}
\medskip

Without loss of generality,
we may assume that $3$ occurs in $S(z_1y_2)$.
There are three possibilities:
\begin{enumerate}[\indent \rm (i)]
\item
$S(z_1y_2)$ consists of only $3$, where
$S(z_1)=(n_1\langle 3 \rangle)$,
$S(y_2)=(n_2\langle 3 \rangle)$,
and $S(z_1y_2)=((n_1+n_2)\langle 3 \rangle)$
with $n_1,n_2\in\ZZ_+\cup\{0\}$;

\item
$S(z_1y_2)$ consists of only $3$, where
$S(z_1)=(n_1\langle 3 \rangle,\ell_1)$,
$S(y_2)=(\ell_2,n_2\langle 3 \rangle)$,
and $S(z_1y_2)=((n_1+n_2+1)\langle 3 \rangle)$
with $\ell_1,\ell_2\in\{1,2\}$, $\ell_1+\ell_2=3$ and
$n_1,n_2\in\ZZ_+\cup\{0\}$;

\item
$S(z_1y_2)$ consists of $2$ and $3$.
\end{enumerate}
By an argument as in Case~1.a(i) and (ii) in Section
~\ref{proof_for_twist_links_2},
we can see that neither (i) nor (ii) can happen.
(In the above, we need to appeal to Lemma~\ref{lem:vertex_position}(2)
instead of Lemma~\ref{lem:vertex_position}(1).)

So, we may assume (iii) holds.
By Lemma~\ref{lem:case3-2}(2) and Remark~\ref{rem:(2)holds}(3),
we see that
$S(z_1y_2)=(n_1 \langle 3 \rangle, 2, 2, n_2 \langle 3 \rangle)$,
where $0 \le n_i \le n-1$ is an integer for $i=1,2$.
We show that $n_1=n_2=n-1$.
Suppose that this does not hold.
If $J=M$, then $CS(\phi(\delta^{-1}))=CS(s')$ includes
both a term $1$ and a term $3$, contradicting \cite[Lemma~3.8]{lee_sakuma_2}.
On the other hand, if $J \subsetneq M$,
then we see, by an argument using Lemma~\ref{lem:vertex_position}(2)
as in Case~2.b,
that a subsequence of the form $(\ell_1, 1, \ell_2)$
with $\ell_1, \ell_2 \in \ZZ_+$ occurs in $S(\phi(e_2'e_3'))$,
so in $CS(\phi(\partial D_1'))=CS(r)$, a contradiction.
Thus we have
$S(z_1y_2)=((n-1) \langle 3 \rangle,2,2,(n-1) \langle 3\rangle)$.
To avoid a contradiction to \cite[Lemma~3.8]{lee_sakuma_2},
we must have $n=2$, namely $r=3/8$.
Again by \cite[Lemma~3.8]{lee_sakuma_2} and the above argument,
we see
$S(z_iy_{i+1})=(3,2,2,3)$
for every $i$.
Hence
$CS(s)=\lp 2,3,2,2,3,\dots,2,3,2,2,3\rp$,
and therefore
$CT(s)=\lp 1,2,\dots,1,2 \rp$.
Thus by using \cite[Lemma~3.8]{lee_sakuma_2} and \cite[Corollary~3.14]{lee_sakuma_2},
we see $CT(s)=\lp 1,2,1,2 \rp$ and
so $CS(s)=\lp 2, 3, 2, 2, 3, 2, 3, 2, 2, 3 \rp$.
Then $s=5/12$, and hence $J$ is obtained from the map in
Figure~\ref{fig.III-case-2b}(a) by reversing
the outer and inner boundaries.
Thus the inner boundary label of $J$ is $\lp 2,1,1,2,1,1 \rp$.
If $J=M$, then $CS(s')=CS(\phi(\delta^{-1}))=\lp 2,1,1,2,1,1 \rp$, and so $s'=3/4$.
This again shows that the loops $\alpha_{5/12}$ and $\alpha_{3/4}$ are homotopic in $S^3-K(3/8)$.
If $J \subsetneq M$,
then, we see by using Lemma~\ref{lem:vertex_position}(2)
as in Case~2.b or as in the proof of Lemma~\ref{lem:case3-2}(1),
that a term $1$ appears in
$CS(r)=CS(3/8)=\lp 3,3,2,3,3,2\rp$, a contradiction.
\qed

\section{Proof of Theorems \ref{main_corollary} and
\ref{main_corollary_2}}
\label{proof_of_corollary}

In order to prove Theorems~\ref{main_corollary} and
\ref{main_corollary_2},
we need the following refinement of \cite[Lemma~4.7]{lee_sakuma_2}
(cf. \cite[Lemma~V.5.2]{lyndon_schupp}).

\begin{lemma}
\label{analogy_lyndon_schupp}
Suppose $G=\langle X \,|\, R \, \rangle$ with $R$ being symmetrized.
Let $u, v$ be two cyclically reduced words in $X$
which are not trivial in $G$,
and let $w$ be a nontrivial reduced word in $X$ such that
$u=wvw^{-1}$ in $G$ but $u \neq wvw^{-1}$ in $F(X)$.
Then the relation $u=wvw^{-1}$ in $G$
is realized by a nontrivial reduced annular $R$-diagram.
To be precise, there is a reduced annular $R$-diagram $(M,\phi)$
and an edge path, $\gamma$, in $M$,
joining a vertex $O_+$ of the outer boundary and a vertex $O_-$
of the inner boundary
which satisfy the following conditions.
\begin{enumerate}[\indent \rm (i)]
\item There is an outer boundary cycle $\alpha$ with base point $O_+$
such that $\phi(\alpha)$ is visibly equal to the word $u$.

\item There is an inner boundary cycle $\delta$ with base point $O_-$
such that $\phi(\delta)$ is visibly equal to the word $v^{-1}$.

\item The word $\phi(\gamma)$ is equal to $w$ in $G$.
\end{enumerate}
\end{lemma}

\begin{proof}
We prove the lemma by imitating
\cite[Proof of Lemma~V.5.2]{lyndon_schupp}.
By the assumption that $u=wvw^{-1}$ in $G$,
$u$ is equal to some product
$pp_1\cdots p_n$ in $F(X)$,
where $p=wvw^{-1}$ and $p_i=c_ir_ic_i^{-1}$ with $r_i\in R$.
We may assume the number $n$ is minimal
among all such products.
Let $M_0'$ be a simply connected diagram over $F(X)$
consisting of $n+1$ disks $D, D_1,\cdots, D_n$
with stems $\gamma,\gamma_1,\cdots, \gamma_n$
joined to a distinguished vertex $O_+$,
satisfying the following conditions
(see \cite[Figure~V.1.1]{lyndon_schupp}).
\begin{enumerate}[\indent \rm (i)]
\item Let $O_-$ be the endpoint of $\gamma$ in $\partial D$.
Then the label of the boundary cycle of $D$ with base point $O_-$
is visibly equal to $v$.

\item The label of the boundary cycle of $D_i$, whose base point is
the endpoint of $\gamma_i$ in $\partial D_i$,
is visibly equal to $r_i$ for each $i\in \{1,\cdots, n\}$.

\item Let $\alpha$ be the boundary cycle of $M_0'$
with base point $O_+$ and with initial segment $\gamma$.
Then the label of $\alpha$ is visibly equal to
the product $pp_1\cdots p_n$.
\end{enumerate}
By applying the operations in
\cite[Proof of Theorem~V.1.1]{lyndon_schupp} to $M_0'$,
obtain a simply connected diagram $M'$
whose boundary label is visibly equal to the reduced word $u$.
Since $u \neq 1$ in $G$,
the $2$-cell $D$ with label $v$ was not deleted
in the construction of $M'$.
Form an annular diagram
$M$ from $M'$ by deleting the (interior of) the $2$-cell $D$.
Since $n$ is minimal, $M$ is reduced
by the argument in \cite[Proof of Lemma~V.2.1]{lyndon_schupp}.
Continue to denote by $O_{\pm}$ the vertices of $M$
determined by the vertices $O_{\pm}$ of $M_0'$.
(It should be noted that both vertices are not removed during the construction.)
Continue to denote by $\gamma$ the edge path in $M$ \lq obtained'
from the edge $\gamma$ of $M_0'$.
(During the construction subsegments of $\gamma$ may be
replaced with homotopic segments.)
Then $\gamma$ joins $O_+$ and $O_-$,
and we have $\phi(\gamma)=w$ in $G$.
Continue to denote by $\alpha$ the outer boundary cycle of $M$
obtained from the outer boundary cycle $\alpha$ of $M_0'$
with base point $O_+$,
and let $\delta$ be the inner boundary cycle of $M$
obtained from the inverse
of the boundary cycle of $D$ in $M_0'$ with base point $O_-$.
Then we see $\phi(\alpha)\equiv u$, $\phi(\delta)\equiv v^{-1}$.
Finally $M$ is nontrivial because
the product $wvw^{-1}$ is not equal to $u$ in $F(X)$.
This completes the proof of Lemma~\ref{analogy_lyndon_schupp}.
\end{proof}

We also need the following fact.

\begin{lemma}
\label{centralizer}
Suppose $r=q/p$, where $p$ and $q$ are relatively prime integers
such that $q\not\equiv \pm1 \pmod{p}$.
Then, for a nontrivial element $u\in G(K(r))$,
the centralizer
$Z(u)=\{w\in G(K(r)) \svert wuw^{-1}=u\}$
of $u$ in $G(K(r))$ is described as follows.
\begin{enumerate}[\indent \rm (1)]
\item If $u$ is non-peripheral, then $Z(u)$ is an infinite cyclic group
generated by some primitive element $u_0$
such that $u=u_0^k$ for some integer $k\ge 1$.

\item If $u$ is peripheral, then $Z(u)$ is conjugate to the
peripheral subgroup, $\langle m, l \rangle \cong \ZZ\oplus \ZZ$,
generated by a meridian and longitude pair $\{m, l\}$.
In particular, $u$ is peripheral if and only if
it commutes with a (conjugate of a) meridian.
\end{enumerate}
\end{lemma}

\begin{proof}
By the assumption,
$K(r)$ is hyperbolic and hence
$G(K(r))$ is identified with a discrete subgroup of $\PSL(2,\CC)$.
Thus the desired fact follows from \cite[Lemma~4.5]{Scott}
\end{proof}

\begin{proof}{\it of Theorems~\ref{main_corollary} and
\ref{main_corollary_2} }
Consider a $2$-bridge link $K(r)$ with
$r=n/(2n+1)$ and $(n+1)/(3n+2)$,
where $n \ge 2$, and
assume that the loop $\alpha_s$ with $s\in I_1(r) \cup I_2(r)$
is either imprimitive or peripheral.
Then, by Lemma~\ref{centralizer},
there is a nontrivial element $w\in G(K(r))$
such that $w\not\in \langle u_s\rangle$
and $wu_sw^{-1}=u_s$.
This identity cannot hold in $F(a,b)$,
since $u_s$ is not a nontrivial cyclic permutation of itself.
So by Lemma~\ref{analogy_lyndon_schupp},
the identity $wu_sw^{-1}=u_s$ in $G(K(r))$
is realized by a nontrivial reduced annular $R$-diagram, $M$,
with outer and inner labels $u_s$ and $u_s^{-1}$, respectively.
Then $M$ satisfies the assumption of \cite[Theorem~4.9]{lee_sakuma_2}
and hence its conclusion.
Thus the arguments in Sections~\ref{proof_for_twist_links_1},
\ref{proof_for_twist_links_2} and \ref{proof_for_[2,1,n]}
reveal all possible shapes of the annular diagram $M$.

\medskip
\noindent {\bf Case 1.} {\it $r=2/5$.}
\medskip

In this case, the arguments in Section~\ref{proof_for_twist_links_1}
imply that one of the following holds.

\begin{enumerate}[\indent \rm (i)]
\item $CS(s)=\lp 5,5 \rp$ and hence $s=1/5$
(see \cite[Remark~3.15]{lee_sakuma_2}).
In this case, $M$ is equivalent to one of the reduced annular diagrams
in Figure~\ref{fig.case-1a} and their natural generalizations.

\item $CS(s)=\lp 4,3,4,3 \rp$ and hence $s=2/7$.
In this case, $M$ is equivalent to one of the reduced annular diagram in
Figure~\ref{fig.case-1c}(b).

\item $CS(s)=\lp 2,1,1,2,1,1 \rp$ and hence
$s=3/4$.
In this case, $M$ is equivalent to one of the reduced annular diagram in
Figure~\ref{fig.case-2b-iii}(b).

\item $CS(s)=\lp 2,2,1,2,2,1 \rp$ and hence
$s=3/5$.
In this case, $M$ is equivalent to one of
the reduced annular diagrams in
Figure~\ref{fig.case-2b-i} and their natural generalizations.
\end{enumerate}

Suppose that condition (i) holds
and that $M$ is as in Figure~\ref{fig.case-1a}(a).
Starting from the uppermost vertex,
read the outer boundary label in the clockwise
direction with the initial label $a$.
Then we obtain the word
\[
aba^{-1}b^{-1}a^{-1}b^{-1}a^{-1}bab
=wu_{1/5}^{-1}w^{-1},
\]
where
$w=b^{-1}a^{-1}b^{-1}$ and
$u_{1/5}=ababab^{-1}a^{-1}b^{-1}a^{-1}b^{-1}$.
Then the inner boundary label with the same base vertex
and with the clockwise direction
is equal to
\[
baba^{-1}b^{-1}a^{-1}b^{-1}a^{-1}ba
=bwu_{1/5}^{-1}w^{-1}b^{-1}
\]
Since these two words determine the same element
of $G(K(2/5))$,
we see that $b$ and $wu_{1/5}^{-1}w^{-1}$ commute with each other.
So, $u_{1/5}$ commutes with $w^{-1} b w$,
which is conjugate to a meridian $b$.
Hence $u_{1/5}$ is peripheral by Lemma~\ref{centralizer}.
The annular diagram in Figure~\ref{fig.case-1a}(b)
shows only that $u_{1/5}$ commutes with $w^{-1} b^2 w$,
and similarly any natural generalization of
the annular diagram in Figure~\ref{fig.case-1a}
imply only that $u_{1/5}$ commutes with a power of $w^{-1} b w$.
Hence the centralizer $Z(u_{1/5})$ is the rank $2$ free abelian group
generated by $u_{1/5}$ and $w^{-1} b w$.
Hence $u_{1/5}$ is primitive.

Suppose that condition (ii) holds.
By reading the annular diagram in Figure~\ref{fig.case-1c}(a)
as above, we see that
$(ab)^{-1}u_{2/7}(ab)$
is equal to $(ba)^3$.
Hence $u_{2/7}=((ab)(ba)(ab)^{-1})^3$ is imprimitive.
On the other hand, the annular diagram in Figure~\ref{fig.case-1c}(b)
shows
\[
(ab)^{-1}u_{2/7}(ab)=(b^2ab^{-1}a^{-1}b^{-1})u_{2/7}(b^2ab^{-1}a^{-1}b^{-1})^{-1}.
\]
Letting $w:=b^2ab^{-1}a^{-1}b^{-1}$, we see that
\[
w=b(bab^{-1}a^{-1}b^{-1})=b(abab^{-1}a^{-1})=(ba)^2(ab)^{-1}\]
(note that the second equality of the above identity comes from
$1=u_{2/5}=abab^{-1}a^{-1}(baba^{-1}b^{-1})$), so that
\[
(ab)^{-1}u_{2/7}(ab)=(ba)^2(ab)^{-1}u_{2/7}(ab)(ba)^{-2},
\]
namely, $(ba)^2$ commutes with $(ab)^{-1}u_{2/7}(ab)$.
Since this diagram is the unique annular reduced diagram
realizing self conjugacies for $u_{2/7}$,
we see that the centralizer $Z(u_{2/7})$ is the infinite cyclic group
generated by $(ab)(ba)(ab)^{-1}$.
So, we can conclude that $u_{2/7}$ is not peripheral
by Lemma~\ref{centralizer}.

Suppose that condition (iii) holds.
By reading the annular diagram in Figure~\ref{fig.case-2b-iii}(a)
as above, we see that
$w^{-1}u_{3/4}^{-1}w$ with $w=aba^{-1}$
is equal to $(b^{-1}a^{-1}ba)^3$.
Hence $u_{3/4}=(w(a^{-1}b^{-1}ab)w^{-1})^3$ is imprimitive.
On the other hand, the annular diagram in Figure~\ref{fig.case-2b-iii}(b)
shows that $a^{-1}b^{-1}ab$ commutes with $w^{-1}u_{3/4}^{-1}w$.
Moreover, by an argument as in (ii),
we see that $Z(u_{3/4})$ is the infinite cyclic group generated by
$w(a^{-1}b^{-1}ab)w^{-1}$.
Hence $u_{3/4}$ is not peripheral by Lemma~\ref{centralizer}.

Suppose that condition (iv) holds.
By reading the annular diagram in Figure~\ref{fig.case-2b-i}(a)
as above, we see that
$u_{3/5}$ is equal to $bu_{3/5}b^{-1}$.
Hence, $u_{3/5}$ is peripheral by Lemma~\ref{centralizer}.
We can also see
that $u_{3/5}$ is primitive
by an argument as in (i).

Since we have checked all possible cases for the 2-bridge knot $K(2/5)$,
the proof of Theorem~\ref{main_corollary} for $K(2/5)$ is complete.

\medskip
\noindent {\bf Case 2.} {\it $r=n/(n+1)$ with $n\ge 3$.}
\medskip

In this case, the arguments in Section~\ref{proof_for_twist_links_2}
imply that one of the following holds.

\begin{enumerate}[\indent \rm (i)]
\item $n=3$ and $CS(s)=\lp 4,3,4,3 \rp$, i.e., $r=3/7$ and $s=2/7$.
In this case,
$M$ is equivalent to the diagram
in Figure~\ref{fig.II-case-1b}(a).

\item $CS(s)=\lp n\langle 2\rangle,1,n\langle 2\rangle,1 \rp$, i.e.,
$s=(n+1)/(2n+1)$.
In this case, $M$ is equivalent to one of the diagrams
in Figure~\ref{fig.II-case-2a-i}
and their natural generalizations.
\end{enumerate}

Suppose that condition (i) holds.
By reading the annular diagram in Figure~\ref{fig.II-case-1b}(a)
as in Case 1(i),
we see that
\[
(ab)^{-1}u_{2/7}(ab)=
(bab^{-1}a^{-1}b^{-1})u_{2/7}(bab^{-1}a^{-1}b^{-1})^{-1}.
\]
So $w:=ab^2ab^{-1}a^{-1}b^{-1}$ belongs to the centralizer $Z(u_{2/7})$.
Since this diagram is the unique annular reduced diagram
realizing self conjugacies for $u_{2/7}$,
we see that $Z(u_{2/7})$ is the infinite cyclic group
generated by $w$.
This implies that $u_{2/7}$ is not peripheral.
On the other hand, we see
\[
w^2=ab(bab^{-1}a^{-1}b^{-1}ab)bab^{-1}a^{-1}b^{-1}
=ab(aba^{-1}b^{-1}a^{-1}ba)bab^{-1}a^{-1}b^{-1}
=u_{2/7}.
\]
In the above identity,
the second equality follows from the relation
\[
1=u_{3/7}=abab^{-1}a^{-1}(bab^{-1}a^{-1}b^{-1}ab)a^{-1}b^{-1}.
\]
Hence $u_{2/7}=w^2$ is imprimitive.

Suppose that condition (ii) holds.
By reading the annular diagram in~\ref{fig.II-case-2a-i}(a)
as in Case 1(i),
we see $u_{s}=b^{-1}u_s b$, where $s=(n+1)/(2n+1)$.
Hence $u_s$ is peripheral.
We can also see as in Case 1(i) that
$u_s$ is primitive.

This completes the proof of Theorem~\ref{main_corollary} for
$K(n/(n+1))$ with $n\ge 3$.

\medskip
\noindent {\bf Case 3.} {\it $r=(n+1)/(3n+2)$ with $n\ge 2$.}
\medskip

In this case, we see from the arguments in
Section~\ref{proof_for_[2,1,n]} that
there is no such annular diagram.
Hence every $\alpha_s$ with $s\in I_1(r) \cup I_2(r)$
is primitive and is not peripheral.
This completes the proof of Theorem~\ref{main_corollary_2}.
\end{proof}

\section*{Acknowledgement}
The authors would like to thank Koji Fujiwara
for stimulating conversations.
They would also like to thank the referee for careful reading
and helpful comments.

\bibstyle{plain}

\end{document}